\documentclass[a4paper,10pt]{article}

\usepackage{amsthm,amsmath,amssymb,pb-diagram}
\usepackage{textcomp,hyperref}
\usepackage{graphicx, graphics, color}
\usepackage[latin1]{inputenc}
\usepackage[automark]{scrpage2}
\usepackage{ifthen}
\usepackage{float}
\usepackage[smalltableaux]{ytableau}

\newtheorem{thr}{Theorem}[section]
\newtheorem*{thr*}{Theorem}
\newtheorem{prp}[thr]{Proposition}
\newtheorem{lmm}[thr]{Lemma}
\newtheorem{crl}[thr]{Corollary}
\newtheorem*{crl*}{Corollary}

\theoremstyle{definition}
\newtheorem{dfn}[thr]{Definition}
\newtheorem{rmr}[thr]{Remark}
\newtheorem{dfnrmr}[thr]{Definition and Remark}
\newtheorem{exm}[thr]{Example}
\newtheorem*{claim*}{Claim}
\newtheorem{formula}[thr]{Formula}

\def\xto{\xrightarrow}

\makeindex

\begin{document}

\title{\textbf{On singularities in $B$-orbit closures of $2$-nilpotent matrices }}

\author{Martin Bender,\\
Bergische Universit\"at Wuppertal.}

\maketitle

\begin{center}
eMail: \verb+mbender@uni-wuppertal.de+
\end{center}

\begin{abstract}
This paper deals with singularities of closures of $2$-nilpotent Borel conjugacy classes in either a $\text{GL}_n$-conjugacy class or in the nilpotent cone of $\text{GL}_n$. In the latter case we construct a resolution of singularities, in the former we show that singularities are rational by applying a result of M. Brion. We reason why this generalizes the result of N. Perrin and E. Smirnov on the rationality of singularities of Springer fiber components in the two-column case. In the case of Borel orbit closures being contained in orbital varieties, we give an alternative version of L. Fresse's recent singularity criterion.
\end{abstract}

\tableofcontents

\section*{Introduction}
Throughout this paper $K$ denotes an algebraically closed field of characteristic zero. By a variety we mean a seperated, integral scheme of finite type over $K$. Let $G$ be a connected reductive group over $K$ and $\mathfrak{g}$ be its Lie algebra. Then $G$ acts on $\mathfrak{g}$ via the adjoint representation. A $G$-orbit $\mathbb{O} = G.x$ is called nilpotent if $\text{ad}(x) : \mathfrak{g} \to \mathfrak{g}$ is nilpotent as a linear map. Now fix a Borel subgroup $B \subseteq G$ and restrict the $G$-action on $\mathbb{O}$ to $B$. A normal $G$-variety $X$ is called spherical if it contains an open $B$-orbit. By a result of F. Knop \cite[Corollary 2.6]{Knop}, this is known to be equivalent to $B$ acting on the normal variety $X$ with finitely many orbits, in arbitrary characteristic. In \cite[Theorem]{Panyushev} D. Panyushev has shown that $B$ acts on $\mathbb{O}$ with finitely many orbits iff $\mathbb{O}$ is of height not greater than $3$. Here the height of $\mathbb{O}$ is the maximum non-zero degree of $\mathfrak{g}$, for the $\mathbb{Z}$-grading of $\mathfrak{g}$ induced by a $\mathfrak{sl}_2$-triple containing $x$ (so $\deg x = 2$). For $G = \text{Gl}_n$ the height of $\mathbb{O}$ is always even. Therefore $\mathbb{O}$ is a spherical homogeneous space iff $\text{ht}(\mathbb{O}) = 2$. This turns out to be equivalent to $x$ being $2$-nilpotent. By the Jordan normal form $\mathbb{O}$ therefore is equal to $\mathbb{O}_k = \{ x \in \mathfrak{gl}_n \mid x^2 = 0 , \quad \text{rk}(x) = k \}$, for some $k$ with $0 \leq 2k \leq   n$.\\

In this paper we are (with the exception of section $4$) interested in the singularities of $B$-orbit closures $Z \subseteq \mathbb{O}_k$. The Bruhat order on $\mathbb{O}_k$ was already determined by M. Boos and M. Reineke \cite{Boos-Reineke} in terms of \textit{oriented link patterns}, extending A. Melnikov's results on $B$-orbits of $2$-nilpotent upper-triangular matrices (cf. \cite{Melnikov}). We obtain the Bruhat order on $\mathbb{O}_k = \text{Gl}_n /C_k$ by considering the set of left cosets $S_n / W(C_k)$, where $W(C_k)$ is the Weyl group of the reductive part of the stabilizer $C_k$ (Corollary \ref{crlBruhatOrder}), following closely D. Panyushev \cite{Panyushev}. We are then able to construct a resolution of singularities of $Z$ (Theorem \ref{thrResolution}). Following \cite{Brion}, this will lead to our main result.
\begin{thr*}\ref{thrRational}
Closures of Borel conjugacy classes in $\mathbb{O}_k$ have rational singularities. In particular, they are normal and Cohen-Macaulay.
\end{thr*}
In section $4$ we extend the resolution from Theorem \ref{thrResolution} and obtain resolutions of singularities for the closures $\mathfrak{Z}=\overline{Z}\subset\mathcal{N}$, where $\mathcal{N}$ is the nilpotent cone of $\mathfrak{gl}_n$. All closures of $B$-conjugacy classes of $2$-nilpotent matrices are of this form.\\
In order to formulate this result in terms of flags and linear maps, denote by $\widetilde{X} \subseteq (G/B)^r$ a \textit{Bott-Samelson variety} of an element $\tau \in S_n$ of Bruhat length $l(\tau)=r$ (cf. Proposition \ref{prpProductBottSamelson}). For a complete flag of $K$-spaces $V^{\bullet}=(V^1 \subset \ldots \subset V^n )\in G/B$, we say that $u \in \mathfrak{gl}_n$ is $V^\bullet$-\textit{compatible}, if $u$ as a linear map fullfills
 \begin{equation*}
 u (V^i ) \subseteq \left\{
\begin{array}{rl}
 \{0\} & \text{if } i= 1 , \ldots , n-k\\
 V^{i-(n-k)} & \text{if } i = n-k+1 ,  \ldots , n 
\end{array}\right. 
\end{equation*}
\begin{thr*}\ref{thrBottSamelsonSpringer}
Let $\tau \in S_n$ be a minimal length representative of the left coset of $W(C_k)$ in $S_n$ corresponding to the $B$-orbit closure $Z \subseteq \mathbb{O}_k$. Further, denote by $r=l(\tau)$ the Bruhat length of $\tau$, and define a closed subvariety of $\mathcal{N} \times \tilde{X}$ by
$$\widetilde{\mathfrak{Z}} = \{ (u,(V_i^{\bullet})_i ) \mid u \text{ is } V_r^{\bullet}-\text{compatible}\}.$$
Then, $\mathfrak{Z}$ is the image of the morphism $\widetilde{\mathfrak{Z}} \to \mathcal{N}$, $(u, (V_i^{\bullet})_i ) \mapsto u$, and
$$\widetilde{\mathfrak{Z}} \to \mathfrak{Z},\quad (u, (V_i^{\bullet})_i ) \mapsto u $$
is a resolution of singularities.
\end{thr*}
The remaining sections of this paper are dedicated to finding singularity criteria for $B$-orbit closures $Z \subseteq\mathbb{O}_k$. We start with relating certain orbit closures in $\mathbb{O}_k$ to fibre bundles on Schubert varieties. By doing so we gain insight into the singular locus of these orbit closures (Proposition \ref{prpSmooth}). For general $Z$ we define $T_k$-\textit{lines} in $Z$ which contain the $T_k$-fixed point $p= \text{id}C_k$ of the minimal $B$-orbit, where $T_k$ denotes a maximal torus in the stabilizer of $\mathbb{O}_k$. Although for general $Z$ this turns out to be insufficient for computing tangent space dimensions (see Example \ref{exm2}), in the case of $Z$ being contained in the affine space of upper-triangular matrices, the number of these lines determines the dimension of the tangent space $T_p (Z)$:
\begin{crl*}\ref{thrUppercaseTangent} Let $Y^0$ be the minimal $B$-orbit in $\mathbb{O}_k$ and $t_k (Z)$ the number of $T_k$-lines in $Z$ through $p$ (cf. Definition \ref{dfnroots}). If $Z$ is contained in the affine space of upper-triangular matrices, then the dimension of the tangent space of $Z$ at $p$ is
$$\dim Y^0 + t_k (Z).$$
\end{crl*}
This result is a translation of L. Fresse's recent result \cite[Theorem 1]{Fresse} on the tangent spaces of certain subvarieties of Springer fiber components in the $2$-column case.\\
$T_k$-lines can be identified with \textit{positive roots} (cf. Definition \ref{dfnroots}) of the stabilizer $C_k$, and them being contained in $Z$ can be checked by using the Bruhat order on $\mathbb{O}_k$ (cf. Proposition \ref{prpcurvecontainment}, part \textit{1.}). As this allows us to compute tangent space dimensions in a way similar to the case of Schubert varieties in $\text{GL}_n / B$, I speculate that there is a \textit{pattern avoidance} singularity criterion. It might be related to the singularity criterion for orbital varieties of $\mathbb{O}_k$ in terms of \textit{minimal arcs of link patterns} given by L. Fresse and A. Melnikov \cite[Theorem 5.2]{FresseMelnikov13}.\\
As well,  we reason why Theorem \ref{thrRational} generalizes a result of N. Perrin and E. Smirnov (\cite{Smirnov-Perrin}) on the Springer fiber components in the $2$-column case (Theorem \ref{thrPerrinSmirnoff}).

\medskip 

\noindent\textbf{Acknowledgements.} First of all I thank Magdalena Boos for stimulating my interest in spherical nilpotent orbits and for explaining to me her work in \cite{Boos}. Furthermore I wish to thank Lucas Fresse and Wiang Yee Ling for providing me with their preprints \cite{Fresse} and \cite{PSY}, and Sascha Orlik for useful comments on an earlier version of this paper.

\section{$B$-orbits in $\mathbb{O}_k$}
Unless stated otherwise, $G$ denotes $\text{GL}_n (K)$, $B \subseteq G$ the Borel subgroup of upper-triangular matrices and $T \subseteq B$ the maximal torus of diagonal matrices. By $W$ we denote the symmetric group on $n$ letters and by $s_i$ the transposition switching $i$ and $i+1$, for $1 \leq i \leq n-1$. $W$ identifies with $N(T)/T$, the Weyl group of $G$. By $\prec$ we denote the usual Bruhat order on $W$. The subgroup of $W$ generated by $\{s_1 , \ldots , s_{k-1} \}$ will be denoted by $W_k$.\\
Once and for all, we fix an ordered $K$-basis $(e_1 , \ldots , e_n )$ of $K^n$. 
\begin{dfn}\label{dfnrankx_k}

\begin{enumerate}
\renewcommand{\labelenumi}{\textbf{(\alph{enumi})}}
\item For $0 \leq 2k \leq n$ we define an element $x_k \in \mathfrak{gl}_n$ by
\begin{equation*}
 x_k (e_i) = \left\{
\begin{array}{rl}
 0 & \text{if } i= 1 , \ldots , n-k\\
 e_{i-(n-k)} & \text{if } i = n-k+1 ,  \ldots , n 
\end{array}\right.
\end{equation*}
\item $\mathbb{O}_k := G . x_k =\{u \in \mathfrak{gl}_n \mid u^2 = 0 , \text{rk}(u)=k \}$, the $G$-conjugacy class of $x_k$. 
\item We denote the stabilizer of $x_k$ in $G$ by $C_k$. 
\end{enumerate}
\end{dfn}
With this notation $\mathbb{O}_k$ can be identified with the quasi-projective variety  $G/C_k$, via the isomorphism
$$G/C_k \to \mathbb{O}_k , \quad gC_k \mapsto g x_k g^{-1} .$$ 

\begin{dfn}\label{associatedFibreBundle} Let $H$ be an subgroup of an algebraic group $A$, acting algebraically on a quasi-projective variety $Z$. Then $H$ acts freely on $A \times Z$ by $h(a,z)  = (ah^{-1},hz)$. The $A$-action on $A \times Z$ given by left-translation on the first component induces an $A$-action on the quotient $A \times_H Z := (A \times Z)/H$. Therefore, $A \times_H Z$ is an $A$-variety. The principal fibre bundle $A \to A/H$ is an \'{e}tale locally trivial $H$-fibration.  
\end{dfn}

\begin{dfn}\label{specialgroup} An algebraic group $A$ is called special if for all subgroups $H \subseteq A$ the principal fibre bundle $A \to A/H$ is a Zariski-locally trivial $H$-fibration.
\end{dfn}

\begin{thr}\label{GLspecial}\cite[Th\'{e}or\`{e}me 1 et 2]{SerreEFA} Special groups are linear and connected. A subgroup $H \subseteq \text{GL}_n$ is special if $\text{GL}_n \to \text{GL}_n / H$ is a Zariski locally trivial $H$-fibration. In particular, $\text{GL}_n$ is special.

\end{thr}

\begin{prp}\label{prps=1}
\begin{enumerate}
\item We have
$$C_k = \left\{ \left(
\begin{array}{ccc}
g & \star & \star \\
0 & h & \star \\
0 & 0 & g 
\end{array}\right) \mid g \in \mathrm{GL}_k , \quad h \in\mathrm{GL}_{n-2k} \right\}.$$ 
Therefore $C_k$ is a closed subgroup of the parabolic subgroup $P_k \supseteq B$ corresponding to the set of simple reflections $\{ s_i \mid i \notin \{k,n-k\} \}$.
\item The morphism
$$\varphi: \mathbb{O}_k \to G/P_k , \quad g C_k \mapsto g P_k $$
is a Zariski locally trivial fibration with fibre $P_k / C_k \simeq \mathrm{GL}_k$.
\item We have a ``Bruhat decomposition''
$$P_k / C_k = \coprod_{\alpha \in W_k} B \dot\alpha C_k  / C_k , $$
where $\dot\alpha$ is a representative of $\alpha$ in $N(T)$. Further, the isomorphism $P_k /C_k \simeq \text{GL}_k$ identifies $(B \dot\alpha C_k )/ C_k$ with $B' \dot\alpha B'$, where $B' \subseteq \text{GL}_k$ is the Borel subgroup of upper-triangular matrices.
\end{enumerate}
\end{prp}
\begin{proof}
\textbf{ad}\textit{ 1.}: We write an element $d \in G$ in the form
$$d=\left(\begin{array}{ccc}
A_1 & H_1 & H_2\\
G_1 & A_2 & H_3\\
G_2 & G_3 & A_3   
\end{array}\right),$$
with $A_1 ,A_3 \in \mathfrak{gl}_k$, $A_2 \in \mathfrak{gl}_{n-2k}$ and suitable matrices $H_i$ and $G_j$. Then the first statement follows from the computations  
$$
d \cdot x_k =
\left(\begin{array}{ccc}
0&0&A_1\\
0&0&G_1\\
0&0&G_2 
\end{array}\right) \text{ and } x_k \cdot 
d =
\left(\begin{array}{ccc}
G_2&G_3&A_3\\
0&0&0\\
0&0&0       
\end{array}\right)$$
\textbf{ad}\textit{ 2.}: $\varphi$ is $G$-equivariant with fibre $P_k / C_k$. So $\varphi$ is isomorphic to the associated fibre bundle $G \times_{P_k} P_k /C_k \to G / P_k$. As $G= \text{GL}_n$ is a special group, $\varphi$ is a Zariski locally trivial $P_k / C_k$-fibration. For the last statement, write an element $p \in P_k$ in the form
$$p= \left( 
\begin{array}{ccc}
p_1 & \star & \star \\
0   & p_2   & \star \\
0   &   0   & p_3 
\end{array}
\right), $$
with $p_1 , p_3 \in \text{GL}_k$, $p_2 \in \text{GL}_{n-2k}$. Then $P_k$ acts on $\text{GL}_k$ by 
$$p.g = p_1 \cdot g \cdot p_3^{-1}.$$
This action is transitive and $C_k = \text{Stab}_{P_k}(\text{id})$, so $P_k /C_k \simeq \text{Gl}_k$.\\ 
\textbf{ad}\textit{ 3.}: As $B \subseteq P_k$, the $P_k$-action on $\text{GL}_k$ restricts to an $B$-action. If $B'$ is the Borel subgroup of upper-triangular matrices in $\text{GL}_k$ and $\alpha\in W_k$, then $B . \dot\alpha = B' \dot\alpha B'$. So $B . \dot\alpha$ is a usual Bruhat cell in $\text{GL}_k$ and we obtain by Bruhat decomposition of $\text{GL}_k$ that
$$P_k / C_k = \coprod_{\alpha \in W_k} B \dot\alpha C_k/C_k .$$
\end{proof}

\begin{rmr}\label{rmrmixed}
$C_k$ is an example for a \textit{mixed subgroup} of $\text{GL}_n$ in the sense of A. Paul, S. Sahi and W. L. Yee (cf. \cite{PSY}): For a reductive group $G$, a parabolic subgroup $P= L U$ with Levi factor $L$ and unipotent radical $U$, consider an involution $\theta \in\text{Aut}(L)$ not necessarily extending to $G$. The subgroup
$$M = L^{\theta} \cdot U , \quad\text{where } L^{\theta} := \{ g \in L \mid \theta(g)=g \},$$
of $P$ is then called a mixed subgroup. In \cite[Proposition 3.3]{PSY} it is shown that $B \backslash G / M$ is finite i.e. mixed subgroups are spherical. In our situation, $P = P_k$ and 
$$\theta : L \to L , \left(
\begin{array}{ccc}
A_1 & 0 & 0\\
0   & A_2 & 0\\
0   & 0   & A_3 
\end{array}\right) \mapsto \left(
\begin{array}{ccc}
A_3 & 0 & 0\\
0   & A_2 & 0\\
0   & 0   & A_1 
\end{array}\right), $$
so it holds that $C_k = L^{\theta} U$.

\end{rmr}

\begin{thr}\label{throrbits} Let $\mathcal{R}\subseteq W$ be any set of representatives of $W/ W(L)$, where $W(L)$ is the Weyl group of the Levi factor of $P_k$. Then the map 
$$\mathcal{C}(-,-): \mathcal{R} \times W_k \mapsto B \backslash \mathbb{O}_k$$
$$(\sigma , \alpha ) \mapsto \mathcal{C}(\sigma , \alpha) := B \dot\sigma\dot\alpha C_k / C_k $$
is a bijection.
\end{thr}

\begin{proof}
It is well-known that 
$$\mathcal{R} \to B \backslash G \slash P_k , \quad \sigma \mapsto B \dot\sigma P_k / P_k$$
is a bijection. In order to finish the proof, it therefore suffices to show that
$$B \dot\sigma P_k /C_k = \coprod_{\alpha \in W_k} B\dot\sigma\dot\alpha C_k /C_k .$$
Let $\rho \in W$. Then $B \dot \sigma B \dot \rho B$ consists of $B \dot\sigma\dot\rho B$ and possibly of some $B \dot\sigma \dot \rho' B$, for certain $\rho' \prec \rho$. Furthermore, we have that $B \dot \rho B C_k / C_k = B \dot \rho C_k / C_k$; identify both $B \dot\rho C_k / C_k$ and $B \dot\rho B C_k / C_k$ with subsets of $\mathbb{O}_k$. We may compute that in both cases we obtain  the set consisting of those $u$ such that there exists $gB \in B \dot\rho B/B$ with
$$u(g(K^{n-k+i})) = g(K^{i}), \quad \text{for all } i=0, \ldots ,k,$$
where $K^j := \langle e_0 , e_1 , \ldots ,e_j \rangle$ and $e_0 := 0$. So, using Proposition \ref{prps=1} \textit{3.}, we can conclude that
$$B \dot\sigma P_k /C_k = \bigcup_{\alpha \in W_k} B\dot\sigma B\dot\alpha C_k /C_k = \bigcup_{\alpha \in W_k} B\dot\sigma \dot\alpha C_k /C_k .$$
For the proof, it remains to be shown that this union is disjoint. We have bijections
\begin{align*}
B\backslash B\dot\sigma P_k / C_k \mathop{\longleftrightarrow}^{1:1} B^{\sigma} \backslash B^\sigma P_k / C_k &\mathop{\longleftrightarrow}^{1:1} B_L^{\sigma} \backslash L / L^{\theta},\\
B \dot\sigma p C_k \longleftrightarrow\quad B^\sigma p C_k &\longleftrightarrow\quad B_L^{\sigma} l_p L^{\theta},
\end{align*}
where we use the notation from Remark \ref{rmrmixed} and where $B^\sigma := \dot\sigma^{-1} B \dot\sigma$, $L^{\theta}$ is the Levi factor of $C_k$, $B_L^{\sigma} := B^\sigma \cap L$ and $p = l_p u_p$, with $l_p \in L$ and $u_p \in U$. Since $B^\sigma \subseteq G$ is a Borel subgroup containing the maximal torus $\dot \sigma^{-1} T \dot \sigma = T$ of $L$, by \cite[Proposition 2.2 , (i)]{Digne-Michel} $B_L^\sigma \subseteq L$ is a Borel subgroup. Therefore $B_L^\sigma$ is conjugated to the standard Borel subgroup $B \cap L \subseteq L$. We then have bijections 
$$B_L^\sigma \backslash L / L^\theta \mathop{\longleftrightarrow}^{1:1} (B \cap L) \backslash L / L^\theta \mathop{\longleftrightarrow}^{1:1} B\backslash P_k / C_k .$$
By composing all bijections, we see with Proposition \ref{prps=1} \textit{3.} that $B \dot \sigma P_k/C_k $ consists of $\mid W_k \mid$ many $B$-orbits of $G/C_k$, finishing the proof.
\end{proof}

\begin{rmr}\label{rmrolps}
Let $(s,r) \in \{1, \ldots ,n\}^2$. Define the elementary matrix $E_{r,s} \in \mathfrak{gl}_n$ by $E_{r,s}(e_k):= \delta_{s,k}e_r$. For a pair $( \sigma , \alpha ) \in \mathcal{R} \times W_k$ we then compute
$$\dot\sigma\dot\alpha C_k = (\dot\sigma\dot\alpha) \cdot x_k \cdot (\dot\sigma\dot\alpha)^{-1} = \sum_{j=1}^k E_{\sigma\alpha(j),\sigma(n-k+j)} \in \mathbb{O}_k .$$
This is the $2$-\textit{nilpotent matrix associated to the oriented link pattern on the $k$ arcs}
$$(\sigma(n-k+1) , \sigma\alpha(1)), \ldots , (\sigma(n),\sigma\alpha(k)), $$
as in \cite{Boos-Reineke}. It was there shown that oriented link patterns on $k$ arcs parametrize $B$-orbits in $\mathbb{O}_k$ by using representation theory of quivers. The term \textit{oriented link pattern} refers to an extension of A. Melnikov's \cite{Melnikov} \textit{link patterns} which give a normal form for $B$-orbits of $2$-nilpotent upper-triangular matrices.  
\end{rmr}

\section{Bruhat order on $B\backslash\mathbb{O}_k$}
We start with some generalities about the (weak) Bruhat order on spherical homogeneous spaces. References for more details are \cite{Brion},\cite{Knop} and \cite{RS}.\\
Denote by $G$ a reductive group and by $B$ a Borel subgroup. For a spherical $G$-homogeneous space $X$ denote by $Z(X)$ the set of $B$-orbit closures in $X$. The \textit{Bruhat order on} $Z(X)$ is defined as the inclusion order of $B$-orbit closures. Denote furthermore by $\Delta$ the set of simple reflections in the Weyl group of $G$. Then the set of minimal parabolic subgroups $\Delta(P)=\{P_{\alpha} \supseteq B \mid \alpha\in\Delta\}$ acts on $Z(X)$. If $P_{\alpha}Z \neq Z$, then $\text{codim}(Z , P_{\alpha}Z)=1$ and one writes $Z \xto{\alpha} P_{\alpha}Z$. The relations $\xto{\alpha}$ generate a partial order on $Z(X)$ called the \textit{weak} Bruhat order. It has been defined and investigated by F. Knop in \cite{Knop}. We will need more results from \cite{Knop} in Section 3 (see Remark \ref{rmrTypesOfEdges}).  

\begin{dfn}\label{dfnz_k} For $\tau\in W$ denote by $l(\tau)$ the Bruhat length of $\tau$. Let
$$Z_k := \{ \sigma \in W \mid l(\sigma s_i) = l(\sigma) + 1 , \quad\forall i \notin \{k,n-k\} \}$$
$$= \{ \sigma \in W \mid l(\sigma \tau) = l(\sigma)+l(\tau), \quad\forall \tau \in W(L) \}, $$
the set of unique minimal length coset representatives of $W/W(L)$. Further, we set $T_k := T \cap C_k$, a maximal torus of $C_k$,  and define
$$W(C_k):= W(L^\theta )= N(T_k )/T_k  \subseteq W(L),$$
the Weyl group of the reductive part of $C_k$ (cf. Remark \ref{rmrmixed}). Then
$$W(L) = \langle s_i \mid i \notin \{k,n-k\} \rangle = \coprod_{\alpha \in W_k} \alpha W(C_k ).$$
\end{dfn}
\begin{rmr}\label{rmrW(C_k)Parametrization}
By Theorem \ref{throrbits}, $W/W(C_k )$ parametrizes the set of $B$-orbits in $\mathbb{O}_k$, as
$$Z_k \times W_k \to W/W(C_k ) , \quad (\sigma , \alpha)\mapsto \sigma\alpha W(C_k )$$
is a bijection. 
\end{rmr}

\begin{lmm}\label{lmmminimal}
A set of minimal length representatives of $W/W(C_k)$ is given by
$$\{ \sigma \alpha \mid (\sigma , \alpha) \in Z_k \times W_k \}.$$
\end{lmm}

\begin{proof}
If $i \notin \{ k , n-k \}$, then $l(\sigma \alpha s_i) = l(\sigma) + l(\alpha s_i)$, since $\sigma\in Z_k$ and $\alpha s_i \in W(L)$. If furthermore $i \geq k+1$ then $l(\alpha s_i) = l(\alpha) + 1$, since $\alpha\in W_k$. It remains to show that $l(\sigma\alpha (s_j s_{n-k+j})) \geq l(\sigma\alpha)$, for all $j \in \{1 , \ldots , k-1 \}$. Because $\sigma \in Z_k$ this is equivalent to $l(\alpha s_j s_{n-k+j}) \geq l(\alpha)$. Now, since $\alpha s_j \in W_k$ one has $l(\alpha s_j s_{n-k+j}) = l(\alpha s_j ) +1 \geq l(\alpha)$. 
\end{proof}

\begin{rmr}\label{rmrBAD}
In general, minimal length representatives of $W/W(C_k )$ are not unique. For example, if $(n,k)=(4,2)$ then $s_1 W(C_k ) = \{ s_1 , s_3 \}$ consists of two elements of Bruhat length $1$.
\end{rmr}

\begin{dfn}\label{dfnZ(s,a)}
Let $(\sigma , \alpha ) \in Z_k \times W_k$ and $s_i \in W$. We denote by
\begin{enumerate} 
\renewcommand{\labelenumi}{\textbf{(\alph{enumi})}}
\item $Z(\sigma , \alpha)\subseteq \mathbb{O}_k$ the closure of $\mathcal{C}(\sigma , \alpha)$, and by
\item $P_i = B \coprod B \dot s_i B \subseteq G$ the parabolic subgroup corresponding to $s_i$.
\end{enumerate}
\end{dfn}

\begin{lmm}\label{lmmclosedorbit}
The $B$-orbit $Y^0 := \mathcal{C}(\text{id},\text{id}) \subseteq \mathbb{O}_k$ is closed.
\end{lmm}

\begin{proof}
One computes that a matrix $u \in \mathbb{O}_k$ belongs to $Y^0 = B .x_k$ if and only if 
$$\langle e_1 , \ldots , e_{n-k} \rangle = \text{Ker}(u) \text{ and } u(\langle e_{n-k+1} , \ldots , e_{n-k+j} \rangle) = \langle e_1 , \ldots , e_j \rangle ,$$
for all $j=1, \ldots ,k$. We conclude that $Y^0$ is closed in $\mathbb{O}_k$, since all matrices in $\mathbb{O}_k$ are of rank equal to $k$.
\end{proof}

\begin{thr}\label{thrBruhat}
Choose for $(\sigma , \alpha ) \in Z_k \times W_k$ reduced expressions\\
$\sigma = s_{i_1} \cdot \ldots \cdot s_{i_r}$ and $\alpha = s_{i_{r+1}} \cdot \ldots \cdot s_{i_l}$. Then the following holds:
\begin{enumerate}
\item $Z(\sigma , \alpha )= P_{i_1}\ldots P_{i_l} Y^0$.
\item $\dim Z(\sigma , \alpha) = l + \dim Y^0$.
\item $Z(\sigma , \alpha) = \bigcup_{ \epsilon \prec \sigma \alpha} B \dot\epsilon C_k / C_k .$
\end{enumerate}
 
\end{thr}

\begin{proof}
Induction on $l = l(\sigma \alpha)$. If $l=0$ then $\sigma \alpha = \text{id}$ and $Z(\sigma , \alpha) = Y^0$ is a closed $B$-orbit. Now assume that the claim is true for all $l' \leq l$, and let 
$$\sigma\alpha = s_{i_1} \ldots s_{i_{l+1}}$$
be a reduced expression of an element $(\sigma , \alpha) \in Z_k \times W_k$. Then, $\beta : = s_{i_2} \cdot \ldots \cdot s_{i_{l+1}}$ is reduced of Bruhat length $l$ because $s_{i_1} \beta = \sigma\alpha$ is reduced of length $l+1$. It is furthermore of the form $\rho \delta$, for a pair $(\rho,\delta) \in Z_k \times W_k$: If $\sigma = \text{id}$ then $\beta \prec \alpha \in W_k$. So $(\rho , \delta)=(\text{id}, \beta) \in Z_k \times W_k$. If $\sigma \neq \text{id}$ then $\beta = \rho \alpha$ where $\rho \in Z_k$. To see this note that for all $i \notin\{k,n-k\}$,
$$l(s_{i_1}\rho s_i) = l(s_{i_1}\rho)+1 = l(\rho) + 2,$$
where the first equality is due to $\sigma = s_{i_1}\rho \in Z_k$ and the second holds due to $s_{i_1} \rho$ being reduced. It follows that $l(\rho s_i ) = l(\rho) +1$ i.e. $\rho \in Z_k$. So we can apply the induction hypotheses:
$$Z(\rho , \delta) = P_{i_2} \cdot \ldots P_{i_{l+1}} Y^0 , \dim Z(\rho , \delta) = l + \dim Y^0 ,\quad Z(\rho ,\delta) = \bigcup_{\epsilon \prec \rho\delta} B \dot\epsilon C_k / C_k .$$
We now show $Z(\sigma , \alpha) = P_{i_1} Z(\rho , \delta )$: We have that 
$$P_{i_1} Z(\rho , \delta) = Z(\rho , \delta ) \cup B\dot{s_{i_1}}Z(\rho , \delta).$$
We conclude that $P_{i_1} Z(\rho , \delta) \neq Z(\rho , \delta ) $, because otherwise $B \dot s_{i_1} Z(\rho ,\delta) \subseteq Z(\rho , \delta)$, and then $\mathcal{C}(\sigma , \alpha) = B \dot \epsilon C_k / C_k$, for some $\epsilon \prec \rho\delta$, which contradicts Lemma \ref{lmmminimal}. We therefore have that $P_{i_1} Z(\rho , \delta ) \subseteq\mathbb{O}_k$ is a $B$-orbit closure of $\dim P_{i_1} Z(\rho , \delta ) =  \dim Z(\rho , \delta) +1$. The open orbit $\mathcal{C}(\sigma' , \alpha')\subseteq P_{i_1} Z(\rho , \delta )$ is contained in $B\dot{s_{i_1}}Z(\rho , \delta)$. Then, due to the induction hypothesis, there exists $\epsilon \prec \rho \delta$ such that 
$$\mathcal{C}(\sigma' , \alpha') \subseteq B \dot s_{i_1} B \dot\epsilon C_k / C_k \subseteq B \dot\epsilon C_k / C_k \cup B \dot s_{i_1}\dot\epsilon C_k / C_k .$$
By Lemma \ref{lmmminimal} it follows that $\sigma' \alpha' = s_{i_1}\rho\delta = \sigma \alpha$, completing the proof. 
\end{proof}

We can now describe the Bruhat order on $B\backslash\mathbb{O}_k$.

\begin{crl}\label{crlBruhatOrder} Let $(\sigma' , \alpha' ),(\sigma , \alpha) \in Z_k \times W_k$. Then $Z(\sigma' , \alpha' ) \subseteq Z(\sigma , \alpha )$ iff there exists $\tau \in W$ such that $\tau \prec \sigma \alpha$ with respect to the ordinary Bruhat order on $W$ and $\overline{\sigma' \alpha'} = \overline{\tau} \in W/W(C(k))$.\\
In this case we use the notation $\sigma' \alpha' W(C_k ) \prec \sigma\alpha W(C_k )$.
\end{crl}
\begin{proof}
$Z(\sigma' , \alpha' ) \subseteq Z(\sigma , \alpha)$ iff $\mathcal{C}(\sigma' , \alpha') \subseteq Z(\sigma , \alpha)$. By Theorem \ref{thrBruhat} \textit{3.}, the latter holds iff there exists $\tau \in W$ such that
$$\mathcal{C}(\sigma' , \alpha' ) = B \dot\tau C_k / C_k \text{ and } \tau \prec \sigma \alpha .$$
By Remark \ref{rmrW(C_k)Parametrization} this is the case iff $\sigma' \alpha'$ and $\tau$ are elements of the same left coset of $W(C_k )$ in $W$. 
\end{proof}

\begin{rmr}\label{rmrLenaDeg}
M. Boos in \cite{Boos} describes the Bruhat order on $B\backslash\mathbb{O}_k$ in terms of oriented link patterns on $k$ arcs. Moreover, she describes as well the Bruhat order on $B\backslash\overline{\mathbb{O}_k}$, where $\overline{\mathbb{O}_k} = \cup_{s \leq k} \mathbb{O}_{s}$ is the closure of $\mathbb{O}_k$ in $\mathfrak{gl}_n$, extending results of A. Melnikov in \cite{Melnikov}. 
\end{rmr}

\section{$B$-orbit closures in $\mathbb{O}_k$ have rational singularities}

Let $Y$ be a variety over $K$ and $f : Z \to Y$ a resolution of singularities. Then the sheaves $R^i f_{*} \mathcal{O}_Z , i \geq 0$, are independent of the choice of the resolution. If $f_{*}\mathcal{O}_Z = \mathcal{O}_Y$ and $R^{i}f_{*}\mathcal{O}_Z = 0$, for all $i > 0$, the variety $Y$ is said to have rational singularities. In this case $Y$ is necessarily normal and Cohen-Macaulay. In this section we show that any $Z(\sigma , \alpha)$ has rational singularities. This will rely on the following Lemma \ref{lmmKey} that states that any $Z(\sigma , \alpha)$ is \textit{multiplicity-free} (cf. \cite{Brion}). The rationality of singularities of $Z(\sigma , \alpha)$ is then part of a result in \cite{Brion}.

\paragraph{Bott-Samelson resolutions of $Z(\sigma, \alpha)$.}

\begin{dfn}\label{dfnRankofOrbit}
Let $G$ be any reductive group and $B\subseteq G$ a Borel subgroup. For an irreducible $B$-variety $X$, we define the \textit{rank of} $X$, $\text{rk}(X)$, as the rank of the abelian group
$$\chi (X)= \{\lambda \in\chi(B ) \mid \exists f \in K(X)^* \quad\forall b \in B : bf = \lambda(b)\cdot f \},$$
where $\chi (B)$ is the group of characters of $B$, and $K(X)$ denotes the field of rational functions on $X$.\\
\noindent Any $B$-orbit $Y$ is isomorphic to $( K^* )^r \times \mathbb{A}^s$, for suitable $r,s \geq 0$. If $Y \subseteq X$ is an open $B$-orbit, then it holds that $\text{rk}(X)= r$ (cf. \cite[Lemma 1]{Brion}).
\end{dfn}

\begin{rmr}\label{rmrRankO_k} By \cite[(4.2), part (a)]{Panyushev} one has 
$$\text{rk}(\mathbb{O}_k) = k,$$
the rank of the matrix $x_k$; by \cite[(1.2), Theorem, part (ii)]{Panyushev} one has 
$$\text{rk}( \mathbb{O}_k ) = \text{rk}_{P_k} (P_k / C_k ).$$
By Proposition \ref{prps=1}, $P_k / C_k \simeq \text{GL}_k$, and the open $B$-orbit in $\text{GL}_k$ equals the open $(B'-B')$-double coset  $B' \omega_k B' \subseteq \text{GL}_k$, where $B' \subseteq \text{GL}_k$ is the Borel subgroup of upper-triangular matrices, and $\omega_k \in W_k$ is the element of maximal Bruhat length. Now $B' \omega_k B' \simeq (K^* )^k \times \mathbb{A}^{\binom{k}{2}}$, so 
$$\text{rk}(\mathbb{O}_k) = \text{rk}_{ P_k } (\text{GL}_k ) = k.$$    
\end{rmr}

\begin{rmr}\label{rmrTypesOfEdges} For a reductive group $G$ and $X $ a spherical $G$-homogeneous space, denote by $Z = \overline{B.x} \subseteq X$ a $B$-orbit closure, and by $P_{\alpha} \supseteq B$ a minimal parabolic group corresponding to a simple root $\alpha$. Then $P_{\alpha} Z = \overline{B.y}$ is again a $B$-orbit closure. Consider the morphism
$$\pi_{\alpha , Z}: P_{\alpha} \times_B Z \to P_{\alpha} Z , \quad [p,x] \mapsto p.x.$$
Then, $\pi_{\alpha , Z}$ restricts to a morphism 
$$p_{Z , \alpha}: P_{\alpha} \times_B B.x \to P_{\alpha} .x .$$
For $P_{\alpha}Z \neq Z$, F. Knop has shown in \cite[Lemma 3.2]{Knop} that one of the three following cases occurs:
\begin{itemize}
\item \textit{Type U}: $P_{\alpha}.x = B.x \coprod B.y$,\\
 $p_{Z , \alpha}$ is an isomorphism and $\text{rk}(P_{\alpha} Z)= \text{rk}(Z)$.
\item \textit{Type T}: $P_{\alpha}.x = Bx \coprod B.y \coprod B.y'$,\\
$p_{Z , \alpha}$ is an isomorphism and $\text{rk}(Z) = \text{rk}(\overline{B.y'}) = \text{rk}(P_{\alpha}Z)-1$.
\item \textit{Type N}: $P_{\alpha}.x = B.x \coprod B.y$,\\
$p_{Z , \alpha}$ is of degree $2$ and $\text{rk}(Z)= \text{rk}(P_{\alpha}Z)-1$.
\end{itemize}  
\end{rmr}

\begin{lmm}\label{lmmKey} Let $(\sigma , \alpha ) \in Z_k \times W_k$ and $i \in \{1, \ldots ,n-1\}$. Then
\begin{enumerate}
\item $\mathrm{rk}(Z(\sigma , \alpha)) = k$.
\item If $P_i Z(\sigma , \alpha) \neq Z(\sigma , \alpha)$, then the morphism
$$\kappa : P_i \times_B Z(\sigma , \alpha ) \to P_i Z(\sigma , \alpha ), \quad [p,g C_k] \mapsto pg C_k $$
is birational. 
\end{enumerate}
\end{lmm}

\begin{proof}
\textbf{ad} \textit{1.}: Let $\sigma\alpha = s_{i_1}\ldots s_{i_r}$ be a reduced expression. By Theorem \ref{thrBruhat} \textit{1.}, we have $Z(\sigma , \alpha) = P_{i_1} \cdot \ldots \cdot P_{i_r}Y^0$. Further, by \cite[Proposition 5, (ii)]{Brion} there exists a sequence $(s_{j_1} \ldots s_{j_k})$ such that $G/C_k = P_{j_1} \cdot\ldots\cdot P_{j_k}Z(\sigma , \alpha)$. With \cite[Theorem 2.3]{Knop} we then conclude that
$$\text{rk}(Y^0 ) \leq \text{rk}(Z(\sigma , \alpha )) \leq \text{rk}(\mathbb{O}_k )=k.$$
Now
$$\mathbb{O}_k \supseteq Y^0 = \left\{ \left(
\begin{array}{ccc}
0&0&a\\
0&0&0\\
0&0&0 
\end{array}\right)\mid a \in B'\right\} \simeq (K^* )^k \times \mathbb{A}_K^{\binom{k}{2}},$$
where $B' \subseteq \text{GL}_k$  is the Borel subgroup of upper-triangular matrices. Therefore $\text{rk}(Y^0 )=k$, and we obtain that $\text{rk}(Z(\sigma , \alpha )) = k$.\\
\textbf{ad} \textit{2.} By Remark \ref{rmrTypesOfEdges}, statement \textit{1.} implies that all edges in the weak Bruhat graph of $\mathbb{O}_k$ \textit{are of Type} $U$, wherefore
$$B \dot s_i B \times_B \mathcal{C}(\sigma , \alpha) \to B \dot s_i \mathcal{C}(\sigma , \alpha) , \quad [p,g C_k] \mapsto pg C_k$$
is an isomorphism. Therefore $\kappa$ is birational.
\end{proof}

\begin{crl}\label{crlIso}
For $(\sigma , \alpha) \in Z_k \times W_k$ one has an isomorphism of varieties
$$\mathcal{C}(\sigma , \alpha) \simeq (K^* )^k \times \mathbb{A}_K^{\binom{k}{2} + l(\sigma) + l(\alpha)} .$$ 
\end{crl}
\begin{proof}
By Lemma \ref{lmmKey} we have $\text{rk}(Z(\sigma , \alpha))=k$. As $\mathcal{C}(\sigma , \alpha)$ is a $B$-orbit, the claim follows now from $\dim Z(\sigma , \alpha) = \dim Y^0 + l(\sigma) + l(\alpha)$. 
\end{proof}

\begin{dfn}\label{dfnBottSamelson}
For $(\sigma , \alpha ) \in Z_k \times W_k$ and a reduced expression $\sigma \alpha = s_{i_1} \cdot \ldots \cdot s_{i_l}$ we define a variety 
$$\tilde{Z} = P_{i_1} \times_B \ldots \times_B P_{i_l} \times_B Y^0 := (P_{i_1} \times\ldots\times P_{i_l} \times Y^0 )/ B^l$$ 
where $B^l$ acts freely on $P_{i_1} \times\ldots\times P_{i_l} \times Y^0$ by
$$\underline{b}.(p_1, \ldots , p_l , bC_k ):=(p_1 b_1^{-1}, \ldots , b_{l-1}p_l b_l^{-1}, b_l b C_k ).$$
Due to Theorem \ref{thrBruhat} the multiplication map
\begin{align*}
\varphi:& \quad\quad\tilde{Z}\to Z(\sigma , \alpha),\\
       &[p_1 , \ldots , p_l , bC_k] \mapsto  p_1 \cdot \ldots \cdot p_l \cdot b C_k , 
\end{align*}
is then a well-defined morphism of varieties.
\end{dfn}
\begin{rmr}\label{rmrBottSamelsonExpression}
$\tilde{Z}$ depends on the choice of a reduced expression for $\sigma \alpha$.
\end{rmr}
\noindent The presence of Lemma \ref{lmmKey} assures that $\varphi$ plays the same role for $Z(\sigma , \alpha)$ as \textit{Bott-Samelson resolutions} play for Schubert varieties. The following result has been pointed out in \cite{Brion}. For the sake of completeness, we give the proof.
\begin{thr}\label{thrResolution}
$\varphi : \tilde{Z} \to Z(\sigma , \alpha)$ is a resolution of singularities. 
\end{thr}
\begin{proof}
\textit{$\tilde{Z}$ is smooth}: Let $l=1$. Then $P_{i_1} \times_B Y^0 \to P_{i_1}/B , [p,bC_k ] \mapsto pB$ is a Zariski locally trivial fibration over $P_{i_1}/B \simeq \mathbb{P}_K^1$ with fiber being the homogeneous space $Y^0$. For $l > 1$ we consider
\begin{align*}
\psi :\tilde{Z} & \to P_{i_1}\times_B \ldots \times_B P_{i_{l-1}}\times_B (P_{i_l}/B)\\ 
[p_1, \ldots , p_l , bC_k ] & \mapsto [p_1 , \ldots ,p_{l-1}, p_l B]
\end{align*}
$\psi$ is a Zariski locally trivial $Y^0$-fibration over $P_{i_1}\times_B \ldots \times_B P_{i_{l-1}}\times_B (P_{i_l}/B)$. The base is locally isomorphic to $(\mathbb{P}^1)^l$. We conclude that $\tilde{Z}$ is a smooth variety.\\
\textit{$\varphi$ is a projective morphism}: $\varphi = \pi \circ \iota$, where $\iota$ is the closed embedding 
\begin{align*}
\tilde{Z} &\to (P_{i_1}\ldots P_{i_l})/B \times Z( \sigma , \alpha),\\
[p_1 , \ldots , p_l ,bC_k ] &\mapsto (p_1 \cdot \ldots \cdot p_l B , p_1 \ldots p_l b C_k )
\end{align*}
and $\pi$ is the projection $(P_{i_1}\ldots P_{i_l})/B  \times Z(\sigma , \alpha) \to Z( \sigma , \alpha)$.\\
Since $(P_{i_1} \ldots P_{i_l})/B = X(\sigma \alpha) \subseteq G/B$ is a Schubert variety, it is projective. We conclude that $\varphi $ is a projective morphism.\\
\textit{$\varphi$ is birational}: Induction on $l = l(\sigma \alpha)$. For $l=1$, $\varphi$ is birational because of Lemma \ref{lmmKey}. For $l > 1$ recall that $s_{i_2} \ldots s_{i_l} = \rho \delta$, for a pair $(\rho , \delta) \in Z_k \times W_k$. Then $\varphi  = \overline{\kappa}\circ \zeta$, where 
\begin{align*}
\zeta :& \tilde{Z} \to P_{i_1} \times_B Z(\rho , \delta)\\
       &[p_1 ,p_2 , \ldots , p_l , bC_k ] \mapsto [p_1 ,p_2 \cdot\ldots\cdot p_l b C_k ]
\end{align*}
is birational by induction hypothesis and 
$$\overline{\kappa}: P_{i_1} \times_B Z(\rho , \delta ) \to P_{i_1} Z(\rho , \delta ), [p_1 , xC_k] \mapsto p_1 xC_k$$
is birational due to Lemma \ref{lmmKey}.
\end{proof}

\begin{rmr}\label{rmrRothbach}
Theorem \ref{thrResolution} was already prooved in \cite{Rothbach} with other methods.  
\end{rmr}

\paragraph{Rationality of singularities.}

\begin{thr}\label{thrRational}
Let $(\sigma , \alpha) \in Z_k \times W_k$. Then $Z(\sigma , \alpha)$ has at most rational singularities. In particular, $Z(\sigma , \alpha)$ is normal and Cohen-Macaulay.  
\end{thr}
\begin{proof}
Lemma \ref{lmmKey} means that $Z(\sigma , \alpha)$ is multiplicity-free. Therefore the claims are part of \cite[Theorem 5]{Brion}.  
\end{proof}

\begin{rmr}\label{rmrRessayre}
 $\mathbb{O}_k$ is an example of what N.Ressayre \cite{Ressayre09} calls a \textit{spherical homogeneous space of minimal rank} equal to $k$, i.e. all $B$-orbit closures are of rank equal to $k$. It follows then from the work of M. Brion \cite{Brion} that $B$-orbit closures in a spherical homogeneous space of minimal rank have Bott-Samelson resolutions, and that their singularities are at most rational. So, the Theorems \ref{thrRational} and \ref{thrResolution} can be extended to the situation of a spherical homogeneous space of minimal rank.    
\end{rmr}

\section{Resolving singularities of $B$-orbit closures in the nilpotent cone}
In this section we extend Theorem \ref{thrResolution} in order to resolve singularities of the closures of $Z(\sigma , \alpha)$ in the nilpotent cone of $\mathfrak{gl}_n$. Afterwards, we describe these resolutions in terms of complete flags and linear maps.
\begin{dfn}\label{dfnBigClosure} For a $B$-orbit closure $Z(\sigma , \alpha)$ in 
$$\mathbb{O}_k = \text{GL}_n . x_k = \{u \in \mathfrak{gl}_n \mid u^2 = 0 , \text{rk}(u) = k \},$$
we define the closure of $Z(\sigma , \alpha)$ in $\mathfrak{gl}_n $, 
$$\mathfrak{Z}(\sigma ,\alpha) := \overline{Z(\sigma , \alpha)} \subseteq \mathfrak{gl}_n .$$
\end{dfn}
$\mathfrak{Z}(\sigma , \alpha)$ is contained in 
$$\overline{\mathbb{O}_k} = \overline{\text{GL}_n . x_k } =  \coprod_{s= 0}^k \text{GL}_n . x_s = \{u ^2 = 0 , \quad \text{rk}(u) \leq k\} =: \mathcal{N}_k^2 .$$
\begin{rmr}\label{rmrTheseAreAll}
Any closure of a $B$-conjugacy classes of a $2$-nilpotent matrix $u$ is of the form $\mathfrak{Z}(\sigma , \alpha)$: For $s = \text{rk}(u)$ we have $B.u \subseteq \mathbb{O}_s$, and according to Theorem \ref{throrbits}, $B . u = \mathcal{C}(\sigma , \alpha)$, for suitable $(\sigma , \alpha) \in Z_s \times W_s$. We therewith obtain that
$$\mathcal{N}_s^2 \supseteq \overline{B.u} = \mathfrak{Z}(\sigma , \alpha).$$ 
\end{rmr}
\begin{dfn}\label{dfnFlagCompability} Let $u \in \mathcal{N}_k^2$ and $V^{\bullet} = (V^1 \subset \ldots \subset V^n ) \in G/B$ be a complete flag of $K$-vector spaces. Recall that $K^\bullet$ denotes the standard flag.
\begin{enumerate}
\renewcommand{\labelenumi}{\textbf{(\alph{enumi})}}
\item $u$ is said to be $V^\bullet$-\textit{compatible} if and only if as a linear map
 \begin{equation*}
 u (V^i ) \subseteq \left\{
\begin{array}{rl}
 \{0\} & \text{if } i= 1 , \ldots , n-k\\
 V^{i-(n-k)} & \text{if } i = n-k+1 ,  \ldots , n 
\end{array}\right.
\end{equation*}
\item The set of $K^\bullet$-compatible elements of $\mathcal{N}_k^2$ we denote by $\mathfrak{Y}^0$.  
\end{enumerate}
Obviously, $\mathfrak{Y}^0$ is stable under $B$-conjugation.
\end{dfn}

\begin{thr}\label{thrBigClosureResolution} Let $(\sigma , \alpha ) \in Z_k \times W_k$, $\sigma \alpha = s_{i_1} \cdot \ldots \cdot s_{i_r}$ a reduced expression and $Y^0$ the minimal $B$-orbit in $\mathbb{O}_k$. Then, the following holds:
\begin{enumerate}
\item $\mathfrak{Y}^0$ is the closure of $Y^0$ in $\mathcal{N}_k^2$.
\item $\mathfrak{Z}(\sigma , \alpha) = P_{i_1} \cdot \ldots \cdot P_{i_r} . \mathfrak{Y}^0 \subseteq \mathcal{N}_k^2$.
\item The multiplication map
\begin{align*}
P_{i_1} \times_B \ldots \times_B P_{i_r} \times_B \mathfrak{Y}^0 &\to \mathfrak{Z}(\sigma , \alpha),\\
[p_1 , \ldots , p_r , u] &\mapsto (p_1 \cdot\ldots\cdot p_r). u ,
\end{align*}
is a resolution of singularities.
\end{enumerate}
\end{thr}
\begin{proof}
\textbf{ad} \textit{1.}: Observe that $Y^0 = \mathfrak{Y}^0 \cap \mathbb{O}_k$, so $Y^0$ is open in $\mathfrak{Y}^0$. Further $\mathfrak{Y}^0$ is irreducible, as it is isomorphic to an affine space. Finally, flag compatibility is a closed condition, hence $\mathfrak{Y}^0 \subseteq \mathcal{N}_k^2$ is closed and the claim follows.\\
In order to proof \textit{2.} and \textit{3.}, we consider the diagram 
$$
\begin{diagram}
\node{P_{i_1}\times_B \ldots P_{i_r} \times_B \mathfrak{Y}^0}\arrow{s,l}{p'}\arrow{e,t}{q'}\node{P_{i_1} \cdot \ldots \cdot P_{i_r} \times_B \mathfrak{Y^0}}\arrow{s,l}{p}\arrow{e,t}{q}\node{\mathcal{N}_k^2}\\
\node{P_{i_1} \times_B \ldots \times_B P_{i_r}/B}\arrow{e,t}{\pi}\node{P_{i_1} \cdot\ldots\cdot P_{i_r}/B = X(\sigma\alpha)}
\end{diagram},
$$
where $p([g,y])=gB$, $q([g,y])= g.y$, $\pi$ is a Bott-Samelson resolution of $X(\sigma\alpha)$ and the square is cartesian.\\
\textbf{ad} \textit{2.}: We start with the following observation:
$$Z(\sigma , \alpha ) = P_{i_1} \cdot \ldots \cdot P_{i_r}. Y^0 = \mathbb{O}_k \cap P_{i_1} \cdot \ldots \cdot P_{i_r}. \mathfrak{Y}^0 .$$
The first equality is Theorem \ref{thrBruhat} and the second equality holds since the rank of a linear map is invariant under conjugation. Therefore $Z(\sigma , \alpha)$ is open in $P_{i_1} \cdot \ldots \cdot P_{i_r}. \mathfrak{Y}^0$. The claim will now follow because $P_{i_1} \cdot\ldots\cdot P_{i_r} . \mathfrak{Y^0}$ is $B$-stable, irreducible and closed in $\mathcal{N}_k^2$. The $B$-stability is obvious. To see irreducibility, note that $p$ is a Zariski locally trivial $\mathfrak{Y}^0$-fibration. Therefore, as $X(\sigma \alpha)$ is irreducible, $P_{i_1} \cdot\ldots\cdot P_{i_r} \times_B \mathfrak{Y}^0$ is as well irreducible. Hence, the image of $q$, $P_{i_1} \cdot \ldots \cdot P_{i_r} . \mathfrak{Y}^0$ is irreducible. Further, $q$ is a projective morphism: $X(\sigma\alpha)$ is a projective variety and $(q,p)$ is an embedding with image being the closed subscheme
$$\mathfrak{X}(\sigma , \alpha):=\{(u,V^{\bullet}) \mid u \text{ is } V^\bullet \text{-compatible}\} \subseteq \mathcal{N}_k^2 \times X(\sigma \alpha).$$
To see this note that a linear map $u$ is $K\bullet$-compatible if and only if $g.u = gug^{-1}$ is $g(K^{\bullet})$-compatible. This means that
$$P_{i_1} \cdot\ldots\cdot P_{i_r} \times_B \mathfrak{Y}^0 \to \mathfrak{X}(\sigma , \alpha),\quad  [g,u] \mapsto (g.u,gB) $$
is an isomorphism of varieties. We conclude that the image of $q$, $P_{i_1} \cdot\ldots\cdot P_{i_r} \mathfrak{Y}^0$, is closed in $\mathcal{N}_k^2$.\\
\textbf{ad} \textit{3.}: We have to show that
$$q \circ q' : P_{i_1} \times_B \ldots \times_B P_{i_r} \times_B \mathfrak{Y}^0 \to P_{i_1} \cdot\ldots\cdot P_{i_r}. \mathfrak{Y}^0$$
is a resolution of singularities. $q \circ q'$ is clearly birational, as it extends the resolution of the open subvariety $Z(\sigma , \alpha)$ from Theorem \ref{thrResolution}. Because the square in the diagram is cartesian and $\pi$ is projective, we deduce that $q'$ is a projective morphism. Since we have seen that $q$ is projective, we conclude that $q \circ q'$ is a projective morphism. As $p'$ is a Zariski locally trivial fibration on a Bott-Samelson variety with smooth fiber $\mathfrak{Y}^0$, we obtain that $P_{i_1} \times_B \ldots \times_B P_{i_r} \times_B \mathfrak{Y}^0$ is a smooth variety.   
\end{proof}
 
\paragraph{$B$-orbit closures as sets of linear maps.} We proceed with describing $\mathfrak{Z}(\sigma , \alpha)$ as a set of linear maps. Denote by $X(\sigma \alpha) \subseteq G/B$ the Schubert variety corresponding to the permutation $\sigma \alpha$.
\begin{dfn}\label{dfnSchubertFlagCompability} Let $u \in \mathcal{N}_k^2$. We say that $u$ is $X(\sigma\alpha)$-compatible if there exists $V^{\bullet} \in X(\sigma\alpha)$ such that $u$ is $V^\bullet$-compatible. Further set
$$\mathfrak{X}(\sigma , \alpha) := \{ (u,V^{\bullet}) \mid u \text{ is } V^\bullet \text{-compatible} \} \subseteq \mathcal{N}_k^2 \times X(\sigma\alpha).$$
\end{dfn}
\begin{lmm}\label{lmmClosureAsSet} The following holds:
\begin{enumerate}
\item $Z(\sigma , \alpha) = \{u \in \mathbb{O}_k \mid u \text{ is } X(\sigma\alpha)-\text{compatible}\}$.
\item $\mathfrak{Z}(\sigma , \alpha) = \{u \in \mathcal{N}_k^2 \mid u \text{ is } X(\sigma\alpha)-\text{compatible}\}$.
\end{enumerate}
\end{lmm}
\begin{proof} The first statement clearly follows from the second one. In the proof of Theorem \ref{thrBigClosureResolution}, \textit{2.}, we have seen that $\mathfrak{Z}(\sigma , \alpha)$ is the image of the morphism $q: P_{i_1} \cdot\ldots\cdot P_{i_r} \times_B \mathfrak{Y}^0 \to \mathcal{N}_k^2$ and that $P_{i_1} \cdot\ldots\cdot P_{i_r} \times_B \mathfrak{Y}^0 \simeq \mathfrak{X}(\sigma , \alpha)$. This proves the second statement.
\end{proof}
\begin{exm}\label{exmEquations} Let $(n,k)=(4,2)$, $\alpha= \text{id}$ and $\sigma = s_1 s_3 s_2 = (2,4,1,3)$. It is well known, that as a set of complete flags, a Schubert variety $X(\tau) \in G/B$ is given by
$$X(\tau) = \{ V^{\bullet} \mid \dim (V^i + K^j) \leq \dim ( E_{\tau}^i + K^j ) , \text{for all } i,j \in [n] \},$$ 
where $K^{\bullet} \simeq \text{id} B$ denotes the standard flag. Therewith, we may compute that a flag $V^{\bullet} = V^1 \subset V^2 \subset V^3 \subset K^4$ belongs to $X(\sigma\alpha)$ iff 
$$V^1 \subset K^2 \subset V^3 .$$
We now may conclude, that
$$Z(s_1 s_3 s_2 , \text{id}) = \{u \in \mathbb{O}_2 \mid u(K^2 ) \subseteq K^2 \}:$$
$\subseteq :$ There exists a flag $V^\bullet$ s.t.r. $V^1 \subset K^2 \subset V^3$ such that $\text{Im}(u) \subseteq V^2 \subseteq \text{Ker}(u)$ and $u(V^3 ) \subseteq V^1$. We obtain that $u^2 = 0$ and $u(K^2 ) \subseteq u(V^3 ) \subseteq V^1 \subseteq K^2$.\\
\noindent $\supseteq$: $u \in\mathbb{O}_2$ if and only if $\text{Im}(u) = \text{Ker}(u)$ and this $K$-space is $2$-dimensional. Further, nilpotency of $u$ implies  $u(K^2 ) \neq K^2$. We distinguish between two cases:
\begin{enumerate}
\item $\dim u(K^2 ) = 0$: Choose $v \in K^4$ with $u(v) = w \neq 0$. Then
$$V^{\bullet}=(\langle w \rangle \subset K^2 \subset K^2 \oplus \langle v \rangle \subset K^4 )\in X(\sigma\alpha),$$
and $u$ is $V^\bullet$-compatible. 
\item $\dim u(K^2 ) =1$: Choose $v \in \text{Im}(u) - K^2$. Then
$$V^{\bullet}=( u(K^2 ) \subset \text{Im}(u) \subset K^2 \oplus \langle v \rangle \subset K^4 ) \in X(\sigma\alpha),$$
and $u$ is $V^\bullet$-compatible.
\end{enumerate}
As $u(K^2 ) \subseteq K^2$ is a closed condition we immediately obtain that
$$\mathfrak{Z}(s_1 s_3 s_2 , \text{id}) = \{u \in \mathcal{N}_2^2 \mid u(K^2 ) \subseteq K^2 \}.$$
\end{exm}
\paragraph{Bott-Samelson-Varieties as subvarieties of $(G/B)^r$.} For more details on the following confer \cite{Magyar}. Let 
$$\tilde{X} = P_{i_1} \times_B \ldots \times_B P_{i_r}$$
be the Bott-Samelson variety associated to $\sigma \alpha \in W$ and $\sigma \alpha = s_{i_1} \ldots s_{i_r}$, the choice of a reduced expression for $\sigma\alpha$. Denote by $p_{i_s} : G/B \to G/ P_{i_s}$ the projection to the partial flag variety $G/P_{i_s}$.
\begin{prp}\label{prpProductBottSamelson}\cite{Magyar} We define a subvariety of $(G/B)^r$ by
$$Y := \{(g_{j}B )_{j \in [r]} \mid g_1 \in P_{i_1} , p_{i_s} (g_s B) = p_{i_s} (g_{s-1}B) , \text{for all } s = 2, \ldots , k \}.$$
Then, $Y$ is the $B$-orbit closure of $B (\dot s_{i_1}B , \ldots , \dot s_{i_1}\cdot\ldots \cdot\dot s_{i_r}B)$, where $B$ acts diagonally on $(G/B)^r$. Moreover, we have a $B$-equivariant isomorphism of varieties
$$\psi: \tilde{X} \to Y , \quad [p_1 , \ldots , p_r] \mapsto (p_1 B , \ldots , p_1 \ldots p_r B).$$
\end{prp}
\begin{rmr}
In the sequel we will use this identification of $\tilde{X}$ as a subvariety of $(G/B)^r$. We have a (product) Bott-Samelson resolution given by
$$\phi : \tilde{X} \to X(\sigma\alpha) , \quad (g_j B)_{j \in [r]} \mapsto g_r B .$$
\end{rmr}
\paragraph{Resolutions for $B$-orbit closures in terms of flags and linear maps.}

\begin{dfn}\label{dfnBottSamelsonSpringer} We define a closed subscheme $\tilde{\mathfrak{Z}}$ of $\mathfrak{Z}(\sigma , \alpha)\times\tilde{X}$ by
$$\tilde{\mathfrak{Z}} := \{ (u,(V_i^{\bullet})_{i \in [r]}) \mid u \text{ is } V_r^{\bullet}-\text{compatible}  \} \subseteq\mathfrak{Z}(\sigma , \alpha)\times\tilde{X}.$$
\end{dfn}
\begin{thr}\label{thrBottSamelsonSpringer}
The morphism 
$$f:\tilde{\mathfrak{Z}} \to \mathfrak{Z}(\sigma , \alpha) , \quad (u,(V_i^{\bullet} )_i ) \mapsto u,$$
is a resolution of singularities.
\end{thr}

\begin{proof} Consider the diagram we dealt with in the proof of Theorem \ref{thrBigClosureResolution}. Using the isomorphisms $P_{i_1} \times_B \ldots \times_B P_{i_r}/B \simeq \tilde{X}$ and $P_{i_1} \cdot\ldots\cdot P_{i_r} \times_B \mathfrak{Y}^0 \simeq \mathfrak{X}(\sigma , \alpha)$ this diagram becomes
$$
\begin{diagram}
\node{\tilde{\mathfrak{Z}}}\arrow{s,l}{q_2}\arrow{e,t}{q_1}\node{\mathfrak{X}(\sigma ,\alpha)}\arrow{s,l}{p_2}\arrow{e,t}{p_1}\node{\mathfrak{Z}(\sigma ,\alpha)}\\
\node{\tilde{X}}\arrow{e,t}{\phi}\node{X(\sigma\alpha)}
\end{diagram},
$$
where $\phi$ is the (product) Bott-Samelson resolution, and $p_1$, $p_2$ are the projections to the first resp. second factor. As in the proof of Theorem \ref{thrBigClosureResolution} we obtain that $\tilde{\mathfrak{Z}}$ is smooth and that $f=p_1 \circ q_1$ is a projective and birational morphism. 
\end{proof}

\begin{rmr}\label{rmrEasyCase} The projective morphism
$$p_1: \mathfrak{X}(\sigma , \alpha) \to \mathfrak{Z}(\sigma , \alpha),\quad (u , V^{\bullet} ) \mapsto u .$$
is birational as well: As $\phi$ is birational, so is $q_1$. Now $p_1 \circ q_1$ is birational, and the birationality of $p_1$ follows. Therefore $p_1$ is a  resolution of singularities if and only if $\mathfrak{X}(\sigma , \alpha)$ is smooth. Because $p_2 :\mathfrak{X}(\sigma , \alpha) \to X(\sigma\alpha)$ is a Zariski locally trivial $\mathfrak{Y}^0$-fibration, the smoothness of $\mathfrak{X}(\sigma , \alpha)$ is equivalent to the smoothness of the Schubert variety $X(\sigma\alpha)$. For instance, this is the case in Example \ref{exmEquations}(in Remark \ref{rmrPattern} we will refer to a singularity criterion for Schubert varieties), and we compute:
\begin{multline*}
\mathfrak{X}(s_1 s_3 s_2 , \text{id})=\{(V^{\bullet} , u) \in G/B \times \mathfrak{gl}_4  \mid \\
 V^1 \subseteq K^2 \subseteq V^3 ,\quad\text{Im}(u) \subseteq V^2 \subseteq \text{Kern}(u),\quad u(V^3 )\subseteq V^1\}.
\end{multline*}
\end{rmr}
\begin{exm}\label{exmSingSchubertNilpotent}
Let $(n,k)=(4,2)$ and $(\sigma , \alpha ) = (s_2 s_1 s_3 s_2 , \text{id})$. Then $\sigma\alpha= (3,4,1,2)$, and we conclude that a flag $V^\bullet$ belongs to the Schubert variety $X(\sigma\alpha)$ if and only if
$$V^1 \subset K^3 \text{ and } K^1 \subset V^3 .$$
Using the method from Example \ref{exmEquations} we may compute that
$$\mathfrak{Z}(\sigma , \alpha) = \{ u \in \mathcal{N}_2^2 \mid u(K^1 ) \subseteq K^3 \}.$$
In Example \ref{exm2} we will see that $\mathfrak{Z}(\sigma , \alpha)$ is singular, and even contains singular points corresponding to matrices of rank $2$.\\
Now $X(\sigma\alpha)$ is singular (cf. Remark \ref{rmrPattern}), hence $\mathfrak{X}(\sigma , \alpha)$ is singular and therewith does not resolve the singularities of $\mathfrak{Z}(\sigma , \alpha)$.\\
The Bott-Samelson variety $\tilde{X}$ associated to the reduced expression $s_2 s_1 s_3 s_2$ is cut out from $(G/B)^4$ by the relations
\begin{align*}
K^1 \subset U \subset  K^3 \subset  K^4 ,\\
W^1 \subset U \subset K^3 \subset K^4 ,\\
W^1 \subset U \subset W^3 \subset K^4 ,\\
W^1 \subset W^2 \subset W^3 \subset K^4 .
\end{align*}
We obtain that $\tilde{\mathfrak{Z}}$ is the variety given by the relations
\begin{align*}
K^1 \subset U \subset  K^3 &,\quad W^1 \subset U \subset  W^3 ,\\
u( W^3 ) \subset W^1 &,\quad \text{Im}(u) \subset W^2 \subset \text{Ker}(u) ,\\
U \in \text{Gr}(2,4)   &, \quad u \in \mathfrak{gl}_4 ,\quad W^{\bullet} \in G/B .
\end{align*}
\end{exm}
\section{Singular $B$-orbit closures in $\mathbb{O}_k$}
This section deals with the question when a $B$-orbit closure $Z(\sigma , \alpha )$ in the homogeneous space $\mathbb{O}_k$ is singular or equivalently, when a closure of a $B$-conjugacy class in $\mathcal{N}_k^2$ has singular points corresponding to matrices of rank $k$. We will \underline{not} find a singularity criterion for $Z(\sigma , \alpha)$ in general. The cases for which we will detect singular $B$-orbit closures are if $\sigma = \text{id}$ or $\alpha =\omega_k$, the longest element of $W_k$. Then, we can relate local properties of $Z(\sigma , \alpha)$ to local properties of the Schubert varieties $X_P (\sigma) \subseteq G/ P_k$ and $X_k (\alpha) \subseteq \text{GL}_k / B'$. We are then done since singularity criteria for Schubert varieties are known (cf. the book \cite{Billey-Lakshmibai}). For a general $B$-orbit closure in $\mathbb{O}_k$, we determine $T_k$-stable lines lying in it and containing the unique $B_k$-fixed point $\text{id}C_k$, where $B_k = B \cap C_k$ and $T_
 k = T \cap C_k$. That is, we mimic an approach used to compute the tangent spaces of Schubert varieties. But, in contrast to the situation of Schubert varieties in $G/B$, the number of these curves in general is strictly smaller than the dimension of the tangent space at $\text{id}C_k$.\\ 

\paragraph{Orbit closures related to Schubert varieties.}

Recall that 
$$\mathbb{O}_k = G \times_{P_k} P_k / C_k \to G/ P_k ,$$
is a Zariski locally trivial $P_k / C_k \simeq \text{GL}_k$-fibration over $G/P_k$ (Proposition \ref{prps=1}).  

\begin{dfn}\label{dfnSchubert}
Let $(\sigma , \alpha) \in Z_k \times W_k$.
\begin{itemize}
\item Define $X_P (\sigma)$ to be the closure of $B \dot\sigma P_k / P_k$ in $G/P_k$, so\\ 
$$X_P (\sigma) = \coprod_{\stackrel{\sigma' \in Z_k}{\sigma' \prec \sigma}} B \dot\sigma' P_k / P_k \subseteq G/P_k .$$
\item $M(\alpha):= \coprod_{\alpha' \prec\alpha} (B\dot\alpha' C_k / C_k) \subseteq P_k / C_k$ identifies with the \textit{quasi-matrix Schubert variety} corresponding to $\alpha$ in $\text{GL}_k$ under the isomorphism $P_k / C_k \simeq \text{GL}_k$ i.e.
$$M(\alpha) \simeq \coprod_{\alpha' \prec \alpha} B' \dot\alpha' B' \subseteq \text{GL}_k ,$$
where $B'$ denotes the Borel subgroup of upper-triangular matrices in $\text{GL}_k$.
\item $X_k (\alpha) := \overline{B' \dot\alpha B' / B'} \subseteq \text{GL}_k /B'$ the Schubert variety corresponding to $\alpha$, being the image of $M(\alpha)$ under the principal bundle $\text{GL}_k \to \text{GL}_k / B'$. 
\item $o_k := (k,\ldots ,1,n-k,\ldots ,k+1,n,\ldots ,n-k+1) \in W(L)$ the longest element in the Weyl group of the Levi factor of $P_k$.
\item $\omega_k = (k , \ldots , 1) \in W_k$ the longest element of $W_k$.
\end{itemize}
\end{dfn}

\begin{prp}\label{prpSmooth} Let $(\sigma , \alpha)  \in Z_k \times W_k$.
\begin{enumerate}
\item If $\alpha = \omega_k$, then  $Z(\sigma , \omega_k )$ is a Zariski locally trivial $\text{GL}_k$-fibration over the Schubert variety $X_P (\sigma) \subseteq G/P_k$. Therefore the singular locus of $Z(\sigma , \omega_k)$ is the union of $B$-orbits $\mathcal{C}(\sigma' , \alpha)$, with $\sigma' P_k \in X_P (\sigma)$ a singular point.
\item If $\sigma = \text{id}$, then $Z(\textit{id} , \alpha)=M(\alpha)$ is a Zariski locally trivial $B'$-fibration over $X_k (\alpha)$. Therefore the singular locus of $Z(\text{id},\alpha)$ is the union of $B$-orbits $\mathcal{C}(\text{id},\alpha)$, with $\alpha' B' \in X_k (\alpha)$ a singular point.
\end{enumerate}
\end{prp}
									
\begin{proof}
\textbf{ad}\textit{ 1.}: $Z(\sigma ,\omega_k )$ is the preimage of $X_P (\sigma)$ under the map $G/ C_k \to G/ P_k$. The restriction $Z(\sigma , \omega_k ) \to X_P (\sigma)$ therefore is a Zariski locally trivial $\text{GL}_k$-fibration. Since $\text{GL}_k$ is smooth as a variety, the claim follows.\\
\textbf{ad}\text{ 2.}: $\mathcal{C}(\sigma , \alpha') \subseteq Z(\text{id},\alpha)$ implies that $B \dot\sigma P_k / P_k \subseteq BP_k /P_k$. Whence $\sigma = \text{id}$ and then $\alpha' \prec \alpha$ by Proposition \ref{prps=1}. This shows $Z(\text{id},\alpha)=M(\alpha)$. The principal $B'$-fibre bundle $\text{GL}_k \to \text{GL}_k /B'$ induces a Zariski locally trivial $B'$-fibration $M(\alpha) \to X_k (\alpha)$. We conclude that the statement on the singular locus of $M(\alpha)$ is true.
\end{proof}

\begin{dfnrmr}\label{dfnrmrMatrixSchubert}
Set $\mathfrak{M}(\alpha):= \mathfrak{Z}(\text{id},\alpha)$. Via the isomorphism $P_k / C_k \simeq \text{GL}_k$, we see that $\mathfrak{M}(\alpha)$ identifies with a closed subvariety $\overline{X_\alpha}$ of $\text{Mat}_k ( K )$, called the \textit{matrix Schubert variety corresponding to the permutation} $\alpha \in S_k = W_k$. Using Lemma \ref{lmmClosureAsSet} we obtain that, as a subset of $\text{Mat}_k (K)$ , $\overline{X_\alpha}$ consists of matrices $A$ such that there exists a flag $(V^1 \subset \ldots \subset V^k ) \in X_k (\alpha)$ with
$$A ( K^i ) \subseteq V^i ,$$
for all $i = 1 , \ldots ,k$. W. Fulton has shown in \cite{Fulton} that $\overline{X_\alpha}$ is Cohen-Macaulay.  
\end{dfnrmr}

\begin{rmr}\label{rmrPattern}
A Schubert variety $X(\sigma)\in\text{GL}_n / B$ is singular if and only if the corresponding element $\sigma \in S_n$ contains sequence $(4,2,3,1)$ or sequence $(3,4,1,2)$ as \textit{a pattern} i.e. there exists $1 \leq i_1 < i_2 < i_3 < i_4 \leq n$ such that
$$\sigma (i_1 ) > \sigma (i_3 ) > \sigma (i_2 ) > \sigma (i_4 ) \text{ or } \sigma (i_2 ) > \sigma (i_1 ) > \sigma (i_4 ) > \sigma (i_3 ).$$
The set $Z_k o_k \subseteq W$ is the set of maximal length representatives of $W/W(L)$. It consists of those $\rho\in W$ fullfilling the property
\begin{align*}
\rho(k)<\ldots < \rho(1),\\
\rho(n-k)<\ldots <\rho(k+1),\\
\rho(n)<\ldots <\rho(n-k+1).
\end{align*} 
Likewise, the set of minimal length representatives of $W/W(L)$, $Z_k$, consists of permutations fullfilling the reversed inequalities from above.
\end{rmr}

\begin{crl}\label{crlk=1}
Let $k=1$. Then, $Z(\sigma , \alpha)\subseteq\mathbb{O}_k$ is singular if and only if $\sigma$ contains $(3,1,4,2)$ as a pattern. 
\end{crl}

\begin{proof}
For $k=1$, $Z(\sigma , \alpha) = Z(\sigma , \text{id}) \to X_P (\sigma )$ is a Zariski locally trivial $K^*$-fibration. Now, $X_P (\sigma)$ is smooth iff $\sigma o_1$ contains neither $(4,2,3,1)$ nor $(3,4,1,2)$ as a pattern. Due to the description of $Z_k o_k$ given in Remark \ref{rmrPattern}, an element of $Z_1 o_1$ can not contain $(4,2,3,1)$ as a pattern. So $Z(\sigma , \alpha)$ is singular iff $\sigma o_1$ contains $(3,4,1,2)$ as a pattern. And this is equivalent to $\sigma$ containing $(3,1,4,2)$ as a pattern. 
\end{proof}

\begin{exm}\label{exm1}
Consider $(n,k)=(4,1)$ and $\sigma= (3,1,4,2)= s_2 s_1 s_3 \in Z_1 $. By Corollary \ref{crlk=1}, $Z:=Z(\sigma , \text{id})\subseteq \mathbb{O}_1$ is singular. Indeed, $\dim Z = l(s_2 s_1 s_3) + \dim Y^0 = 4$ and using Lemma \ref{lmmClosureAsSet} we may compute 
$$Z = \{u \in \mathbb{O}_1 \mid K^1  \subset \text{Ker}(u) ,\quad \text{Im}(u) \subset K^3 \}.$$
Now, set $I = \{(1,2),(1,3),(1,4),(2,4),(3,4)\}$. Then, for any $(i,j) \in I$ the affine line 
$$\mathcal{L}_{i,j} := \{x_1 + t E_{i,j} \mid t \in K \} \simeq \mathbb{A}_K^1$$
is a subvariety of $Z$, and taking derivatives at $t=0$ we obtain that
$$T_{x_1}(Z) \supseteq \sum_{(i,j) \in I} T_{x_1} (\mathcal{L}_{i,j}) \simeq \bigoplus_{(i,j)\in I} K \cdot E_{i,j} = K^5 ,$$
wherefore $Z$ is singular in $x_1$.
\end{exm}

\paragraph{$T_k$-lines in orbit closures.} The motivation for this paragraph comes from the following: Consider a Schubert variety $X(\tau) \subseteq G/B$. As $X(\tau)$ is $B$-stable, it is singular if and only if it is singular in the unique $B$-fixed point $x:= \text{id}B$. Set $V(\tau):= T_x (X(\tau))$ and $V := T_x (G/B)$. Then $V(\tau)$ is a $T$-subrepresentation of $V$, as $x$ is fixed by $T$. One may show that $V \simeq T_{\text{id}}(U^{-}) =\mathfrak{u}^-$, where $U^{-} \subseteq G$ is the subgroup of lower-triangular unipotent matrices, $\mathfrak{u}^-$ is the underlying vector space of the Lie algebra of $U^-$, and $T$ acts on $\mathfrak{u}^-$ as
$$t. A := t \cdot A \cdot t^{-1} , \quad t \in T , A \in \mathfrak{u}^{-}.$$
One concludes that the decomposition of $V$ in $T$-eigenspaces is
$$V =\bigoplus_{\epsilon \in \phi^{+}} V_{-\epsilon}, \quad V_{-\epsilon} = \langle E_{j,i} \rangle ,$$
where $\phi^+$ is the set of positive roots with respect to $B$, and $\epsilon : T \to \mathbb{G}_m , \epsilon (t) = t_i \cdot t_j^{-1}$. On the other hand we have the $1$-dimensional $T$-orbit
$$C_\epsilon^0 := T (\text{id} + E_{j,i})B/B \subseteq G/B .$$
The closure $C_\epsilon := \overline{C_\epsilon^0}\subseteq G/B$ is a projective line obtained by adding the $T$-fixed points $x$ and $\dot r_\epsilon B$ to $C_\epsilon^0$. Here, $r_\epsilon \in W$ is the transposition switching $i$ with $j$. So, the tangent space $T_\epsilon :=T_x (C_\epsilon )$ is a $1$-dimensional $T$-representation. It turns out to be isomorphic to $V_{-\epsilon}$. Further, it can be seen that $C_\epsilon \subseteq X(\tau)$ if and only if $r_\epsilon \prec \tau$. All together, we derive the containment of $T$-representations
$$\bigoplus_{r_\epsilon \prec \tau} V_{-\epsilon} \subseteq V(\tau).$$
This is indeed an equality of $T$-representations. For a proof we refer to \cite[Theorem 3.4]{LakshmibaiTangent} or \cite[Theorem I]{Ryan}. Consequently, $X(\tau)$ is smooth if and only if
$$l(\tau) = \mid \{\epsilon \in \phi^{+} \mid r_\epsilon \prec \tau  \} \mid . $$
It is natural to ask, whether there is an analogous way to compute tangent space dimensions of $B$-orbit closures in $\mathbb{O}_k =G/C_k$. 
\begin{rmr}\label{rmrBigTorusNoFixed}
If $k \neq 0$, the maximal torus $T \subseteq B$ acts on $G/C_k$ without fixed points: A point $g C_k$ is fixed by $T$ if and only if $g^{-1} T g \subseteq C_k$. Therefore the reductive ranks of $G$ and $C_k$ coincide, as $g^{-1} T g $ is a maximal torus of $G$. Because the reductive part of $C_k$ is isomorphic to $\text{GL}_k \times \text{GL}_{n-2k}$, we obtain
$$n=\text{rk}( \text{GL}_n ) = \text{rk}(\text{GL}_k \times \text{GL}_{n-2k}) = n-k,$$
so $k = 0$. Then $\mathbb{O}_k = G . x_0 = \{0\}$, which is an uninteresting case. We therefore rather consider the action of $T_k = T \cap C_k$ on $G/C_k$. Then, $p:= \text{id}C_k$ is fixed by $T_k$, so the tangent space $T_p (Z(\sigma , \alpha))$ is a $T_k$-representation. The point $p$ lies in the minimal $B$-orbit of $G/C_k$, so a $B$-orbit closure $Z(\sigma , \alpha)$ is singular if and only if $p \in Z(\sigma , \alpha)$ is a singular point. 
\end{rmr}

\begin{rmr}\label{rmrTangentSpaceY^0}
The group $T_k$ acts on $\text{End}(K^n)$ via conjugation. As $G/C_k =\mathbb{O}_k \subseteq \text{End}(K^n ) \simeq \mathbb{A}^{n^2}$ is locally closed, we may identify $T_p ( G / C_k ) = T_{x_k}(\mathbb{O}_k )$ with a $T_k$-subrepresentation of
$$T_{x_k} (\mathbb{A}^{n^2}) \simeq  \langle E_{i,j} \mid (i,j) \in [n]^2 \rangle ,$$
where $T_k$ acts via
$$t . E_{i,j} := t E_{i,j} t^{-1} = t_i t_j^{-1} E_{i,j}.$$ 
For the tangent space at $p$ of the minimal $B$-orbit of $\mathbb{O}_k$ we have that
$$T_p (Y^0 ) \simeq \mathfrak{Y}^0 = \langle E_{r,s} \mid (r,s) \in \{1, \ldots ,k\}\times\{n-k+1 , \ldots ,n\} \rangle ,$$
because $Y^0$ is open in the vector space $\mathfrak{Y}^0 \subseteq \text{End}(K^n )$. As $Y^0$ is contained in any $B$-orbit closure $Z$, we are only interested in the numbers
$$\text{codim}(T_p (Y^0) , T_p (Z)).$$
\end{rmr}

\begin{dfn}\label{dfnroots}
Define $\phi^+$ as the set of positive roots of $G$ with respect to $B$. We can identify $\phi^+$ with $\{ (i,j) \mid 1 \leq i < j \leq n \}$. $\epsilon \in \phi^+$ corresponds to a transposition in $W$ which we denote by $r_{\epsilon}$. Further, $\epsilon$ defines two unipotent subgroups $U_{\pm\epsilon}\subseteq G$, which both are isomorphic to the additive group $(\mathbb{G}_a , +)$:
$$u_{\pm \epsilon} : (\mathbb{G}_a , +) \xto{\simeq} U_{\pm\epsilon}, \quad t \mapsto \text{id} + t E_{\pm\epsilon},$$
where $E_{+\epsilon} = E_{i,j}$ and $E_{-\epsilon} = E_{j,i}$. Now, set 
\begin{itemize}
\item $\phi^+ (\text{GL}_k ) = \{ (i,j) \mid 1 \leq i<j \leq k \} \subseteq \phi^+$,
\item $\phi^+ (P_k ) = \{\epsilon\in \phi^+ \mid U_{-\epsilon} \nsubseteq P_k \}$,
\item $\phi^+ (C_k ) := \phi^+ (\text{GL}_k ) \coprod \phi^+ (P_k)$, the set of \textit{positive roots of} $C_k$.
\end{itemize}
\end{dfn}

\begin{dfn}\label{T_kLines}
For $\epsilon \in \phi^+ (C_k )$ we define a $T_k$-stable subset of $G/C_k$ by
$$C(\epsilon) := \{ u_{-\epsilon}(t)C_k \mid t \in K \}.$$
The next Proposition suggests to call them $T_k$-\textit{lines}.
\end{dfn}

\begin{prp}\label{prpcurvecontainment} Let $\epsilon \in \phi^+ (C_k )$.
\begin{enumerate}
\item The set $C(\epsilon)$ is contained in $Z(\sigma , \alpha)$ if and only if 
$$ r_{\epsilon} W(C_k) \prec \sigma\alpha W(C_k).$$
\item Each $C(\epsilon)$ is isomorphic to an affine line in $\mathbb{O}_k$. Explicitly, isomorphisms are given as follows:
\begin{itemize}
\item For $\epsilon =(i,j) \in \phi^+ (\mathrm{GL}_k )$:
$$\mathbb{A}^1 \xto{\simeq} C(\epsilon), \quad t \mapsto x_k +t E_{j,i+n-k},$$
so $T_p (C(\epsilon)) = <E_{j,i+n-k}>$.
\item For $\epsilon = (i,i+n-k) \in \Delta:=\{(1,n-k+1), \ldots , (k,n)\}\subseteq \phi^+ (P_k )$:
$$\mathbb{A}^1 \xto{\simeq} C(\epsilon), \quad t \mapsto x_k + t E_{i+n-k , i+n-k} - t E_{i,i} - t^2 E_{i+n-k , i},$$
so $T_p (C(\epsilon)) = < E_{i+n-k,i+n-k}-E_{i,i} >$.
\item For $\epsilon = (i,j) \in (\{1,\ldots , k\} \times \{n-k+1 , \ldots , n\}) - \Delta \subseteq\phi^+ (P_k )$:
$$\mathbb{A}^1 \xto{\simeq} C(\epsilon), \quad t \mapsto x_k + t E_{j,i+n-k} - t E_{j-n+k , i},$$
so $T_p (C(\epsilon)) = <E_{j,i+n-k}-E_{j-n+k,i}>$.
\item For $\epsilon = (i,j) \in\{1,\ldots ,k\}\times\{k+1 , \ldots , n-k\} \subseteq \phi^+ (P_k )$:
$$\mathbb{A}^1 \xto{\simeq} C(\epsilon), \quad t \mapsto x_k + t E_{j,i+n-k},$$
so $T_p (C(\epsilon)) = <E_{j,i+n-k}>$.
\item For $\epsilon = (i,j) \in\{k+1 , \ldots , n-k\}\times \{n-k+1,\ldots ,n\}\subseteq \phi^+ (P_k )$:
$$\mathbb{A}^1 \xto{\simeq} C(\epsilon), \quad t \mapsto x_k - t E_{j-n+k,i},$$
so $T_p (C(\epsilon)) = <E_{j-n+k,i}>$.
\end{itemize}
\end{enumerate}
\end{prp}

\begin{proof}
\textbf{ad} \textit{1.}: Clearly, $u_{-\epsilon}(0)C_k = \text{id}C_k$ is contained in any $Z(\sigma , \alpha)$. For $t \neq 0$ we compute that
$$u_{-\epsilon}(t)=(\dot r_{\epsilon} + t E_{j,j} - t^{-1}E_{i,i} - E_{j,i})\dot r_{\epsilon} u_{\epsilon}(t^{-1}) \in B \dot r_{\epsilon} B .$$
Consequently, $u_{-\epsilon}(t)C_k \in B \dot r_{\epsilon} B C_k /C_k = B \dot r_{\epsilon} C_k / C_k$. The claim now follows from the Corollary \ref{crlBruhatOrder} describing the Bruhat order on $B \backslash \mathbb{O}_k$. \\
\textbf{ad} \textit{2.}: Using the isomorphism $G/ C_k \to \mathbb{O}_k ,\quad gC_k \mapsto g x_k g^{-1}$, we obtain for arbitrary $\epsilon \in \phi^+$ the equation
$$
u_{-\epsilon}(t)C_k \simeq u_{-\epsilon}(t)x_k u_{-\epsilon}(-t)= x_k + t E_{j,i}x_k - t x_k E_{j,i}-t^2 E_{j,i}x_k E_{j,i}.
$$
Now, use $x_k = \sum_{r=1}^k E_{r,r+n-k}$ and note that $\phi^+ (C_k )$ decomposes into the five specified sets. The descriptions of the tangent spaces $T_p (C(\epsilon))$ are obtained by taking differentials at $t=0$ of the isomorphisms $\mathbb{A}^1 \to C(\epsilon)$.
\end{proof}

\begin{rmr}
$C(\epsilon)$ is the closure of the $1$-dimensional $T_k$-orbit $T_k u_{-\epsilon}(1) C_k / C_k$, if $\epsilon \in \phi^+ (C_k ) - \Delta$. $C(\epsilon)$ consists entirely of $T_k$-fixed points, if $\epsilon \in \Delta$. 
\end{rmr}

\begin{dfn}\label{dfnorbitcurves} For $(\sigma , \alpha) \in Z_k \times W_k$ define
$$t_k (\sigma , \alpha ) := \{ \epsilon \in \phi^+ (C_k ) \mid  r_{\epsilon} W(C_k ) \prec \sigma\alpha W(C_k )\}.$$ 
Further, we define a Borel subgroup of $C_k$ by $B_k := B \cap C_k \subseteq C_k$.
\end{dfn}

\begin{crl}\label{crltangent_T_k}
\begin{enumerate}
\item As a $T_k$-representation $T_p (G/C_k )$ decomposes into $T_k$-subrepresentations
$$T_p (G/C_k ) = T_p (Y^0) \oplus \bigoplus_{\epsilon \in \phi^+ (C_k )} T_p (C(\epsilon)).$$
\item One has that
$$T_k (\sigma , \alpha):= T_p (Y^0) \oplus \bigoplus_{\epsilon\in t_k (\sigma , \alpha)} T_p (C(\epsilon)) \subseteq T_p (Z(\sigma , \alpha))$$
is a $T_k$-subrepresentation. 
\end{enumerate}
\end{crl}

\begin{proof}
As $G/C_k$ is homogeneous, it is smooth hence $\dim T_p (G / C_k) = \dim G - \dim C_k$. With Proposition \ref{prpcurvecontainment}, \textit{2.}, one checks that the sum over all vector spaces $T_p (C(\epsilon))$, $\epsilon \in \phi^+ (C_k )$, is direct. From the description of $T_p (Y^0 )$ given in Remark \ref{rmrTangentSpaceY^0} we obtain that the vector space sum on the RHS is direct. Now the claim follows, since $\phi^+ (C_k )$ has cardinality $\dim G - \dim C_k - \dim Y^0$.\\
The second statement follows now from Proposition \ref{prpcurvecontainment}, \textit{1.}  
\end{proof}

\begin{rmr}\label{rmrT_kWeights}
As $T_p (G/C_k )$ is a $T_k$-representation it decomposes into $T_k$-eigenspaces. If a $T_k$-eigenspace is non-zero, it is not $1$-dimensional in general: For a weight $\lambda : T_k \to K^* , \underline{t} \mapsto t_j t_i^{-1}$ with $1 \leq i \neq j \leq k$ we derive from Proposition \ref{prpcurvecontainment} that
$$(T_p (G/C_k ))_\lambda = \langle E_{j,i+n-k} \rangle \oplus T_p (C(i,j+n-k)) \simeq K^2 .$$
The $T_k$-eigenspace of the trivial character $\lambda = 1_{ T_k }$ is
$$(T_p (G/C_k ))_{1_{T_k}} = \langle E_{r,r+n-k} \mid r \in [k] \rangle \oplus \bigoplus_{\epsilon \in \Delta} T_p (C(\epsilon)) \simeq K^{2k}.$$  
\end{rmr}

\begin{rmr}\label{rmrmodulestructure}
Since $p$ is a $B_k$-fixed point and $Z(\sigma , \alpha)$ is $B$-stable hence $B_k$-stable, the algebraic group $B_k$ acts on the tangent space $T_p (Z(\sigma , \alpha))$. 
\end{rmr}

\begin{dfn}\label{dfnCurve_T_k-B_k}
The $T_k$-subrepresentation $T_k (\sigma , \alpha) \subseteq T_p (Z(\sigma , \alpha))$ generates a $B_k$-subrepresentation of $T_p (Z(\sigma , \alpha))$, which we denote by $B_k (\sigma , \alpha)$.
\end{dfn}
The next example shows that in general $B_k (\sigma , \alpha) \neq T_k (\sigma , \alpha)$.
\begin{exm}\label{exm2} Let $(n,k)=(4,2)$. We have 
$$W_k = \langle s_1 \rangle, \quad Z_k = \{\sigma \in W \mid \sigma(1)<\sigma(2),\sigma(3)<\sigma(4)\},\quad W(C_k )= \langle s_1 s_3 \rangle .$$
The next figure shows the Bruhat graph of $B \backslash \mathbb{O}_2$.
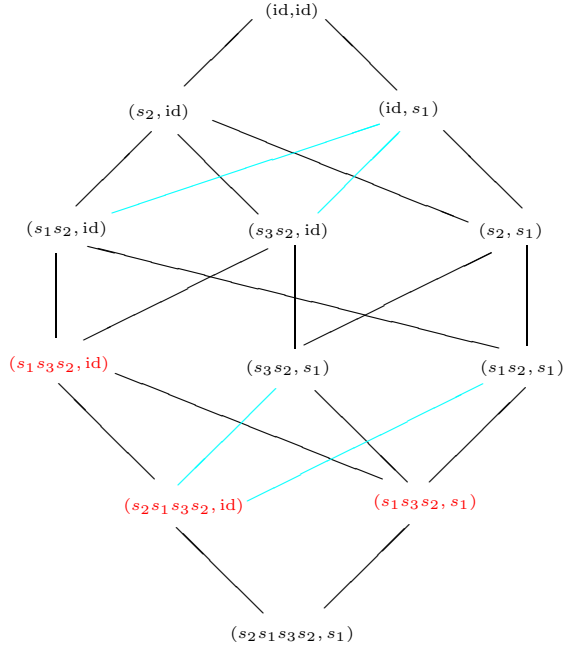
\begin{figure}[H]
\begin{center}
\setlength{\unitlength}{2569sp}%
\begingroup\makeatletter\ifx\SetFigFont\undefined%
\gdef\SetFigFont#1#2#3#4#5{%
  \reset@font\fontsize{#1}{#2pt}%
  \fontfamily{#3}\fontseries{#4}\fontshape{#5}%
  \selectfont}%
\fi\endgroup%
\begin{picture}(4977,6178)(2026,-6611)
\thinlines
{\color[rgb]{0,0,0}\put(5056,-6271){\line( 1, 1){810}}
}%
{\color[rgb]{0,0,0}\put(4471,-6316){\line(-1, 1){825}}
}%
{\color[rgb]{0,0,0}\put(3436,-5026){\line(-1, 1){915}}
}%
{\color[rgb]{0,0,0}\put(2491,-3661){\line( 0, 1){810}}
}%
{\color[rgb]{0,0,0}\put(2686,-2401){\line( 1, 1){720}}
}%
{\color[rgb]{0,0,0}\put(3721,-1246){\line( 1, 1){645}}
}%
{\color[rgb]{0,0,0}\put(6061,-5056){\line( 1, 1){915}}
}%
{\color[rgb]{0,0,0}\put(6991,-3736){\line( 0, 1){960}}
}%
{\color[rgb]{0,0,0}\put(6946,-2416){\line(-1, 1){750}}
}%
{\color[rgb]{0,0,0}\put(3661,-1711){\line( 1,-1){765}}
}%
{\color[rgb]{0,0,0}\put(3991,-1576){\line( 5,-2){2438.793}}
}%
{\color[rgb]{0,0,0}\put(6721,-3766){\line(-4, 1){3924.706}}
}%
{\color[rgb]{0,0,0}\put(4861,-3736){\line( 2, 1){1788}}
}%
{\color[rgb]{0,0,0}\put(4771,-2776){\line( 0,-1){990}}
}%
{\color[rgb]{0,0,0}\put(4516,-2806){\line(-2,-1){1758}}
}%
{\color[rgb]{0,0,0}\put(4966,-4171){\line( 1,-1){885}}
}%
{\color[rgb]{0,0,0}\put(3061,-4006){\line( 5,-2){2550}}
}%
{\color[rgb]{0,1,1}\put(4591,-4156){\line(-1,-1){930}}
}%
{\color[rgb]{0,1,1}\put(4321,-5236){\line( 2, 1){2250}}
}%
{\color[rgb]{0,1,1}\put(4996,-2461){\line( 1, 1){780}}
}%
{\color[rgb]{0,1,1}\put(3031,-2461){\line( 3, 1){2551.500}}
}%
{\color[rgb]{0,0,0}\put(5746,-1276){\line(-1, 1){675}}
}%
\put(3181,-1531){\makebox(0,0)[lb]{\smash{{\SetFigFont{6}{7.2}{\familydefault}{\mddefault}{\updefault}{\color[rgb]{0,0,0}($s_2 ,\text{id}$)}%
}}}}
\put(6526,-2671){\makebox(0,0)[lb]{\smash{{\SetFigFont{6}{7.2}{\familydefault}{\mddefault}{\updefault}{\color[rgb]{0,0,0}($s_2 ,s_1$)}%
}}}}
\put(4321,-2671){\makebox(0,0)[lb]{\smash{{\SetFigFont{6}{7.2}{\familydefault}{\mddefault}{\updefault}{\color[rgb]{0,0,0}($s_3 s_2 ,\text{id}$)}%
}}}}
\put(2206,-2656){\makebox(0,0)[lb]{\smash{{\SetFigFont{6}{7.2}{\familydefault}{\mddefault}{\updefault}{\color[rgb]{0,0,0}($s_1 s_2 ,\text{id}$)}%
}}}}
\put(4306,-3991){\makebox(0,0)[lb]{\smash{{\SetFigFont{6}{7.2}{\familydefault}{\mddefault}{\updefault}{\color[rgb]{0,0,0}($s_3 s_2 ,s_1$)}%
}}}}
\put(6541,-3991){\makebox(0,0)[lb]{\smash{{\SetFigFont{6}{7.2}{\familydefault}{\mddefault}{\updefault}{\color[rgb]{0,0,0}($s_1 s_2 ,s_1$)}%
}}}}
\put(5521,-5281){\makebox(0,0)[lb]{\smash{{\SetFigFont{6}{7.2}{\familydefault}{\mddefault}{\updefault}{\color[rgb]{1,0,0}($s_1 s_3 s_2 ,s_1$)}%
}}}}
\put(5566,-1501){\makebox(0,0)[lb]{\smash{{\SetFigFont{6}{7.2}{\familydefault}{\mddefault}{\updefault}{\color[rgb]{0,0,0}($\text{id},s_1$)}%
}}}}
\put(2041,-3961){\makebox(0,0)[lb]{\smash{{\SetFigFont{6}{7.2}{\familydefault}{\mddefault}{\updefault}{\color[rgb]{1,0,0}($s_1 s_3 s_2 ,\text{id}$)}%
}}}}
\put(3136,-5326){\makebox(0,0)[lb]{\smash{{\SetFigFont{6}{7.2}{\familydefault}{\mddefault}{\updefault}{\color[rgb]{1,0,0}($s_2 s_1 s_3 s_2 , \text{id}$)}%
}}}}
\put(4156,-6556){\makebox(0,0)[lb]{\smash{{\SetFigFont{6}{7.2}{\familydefault}{\mddefault}{\updefault}{\color[rgb]{0,0,0}($s_2 s_1 s_3 s_2 ,s_1$)}%
}}}}
\put(4486,-556){\makebox(0,0)[lb]{\smash{{\SetFigFont{6}{7.2}{\familydefault}{\mddefault}{\updefault}{\color[rgb]{0,0,0}($\text{id,id}$)}%
}}}}
\end{picture}%
\\
\caption{Bruhat graph for $(n,k)=(4,2)$} 
\end{center}
\end{figure}
Edges that are drawn in light-blue correspond to minimal degenerations 
$$\sigma' \alpha' W(C_k ) \prec \sigma \alpha W(C_k )$$
with $\alpha'$ not smaller than $\alpha$ with respect to the Bruhat order on $W$. The existence of these edges shows that  $Z(\sigma , \alpha)$ is not a $M(\alpha)$-fibration over $X_P (\sigma)$ in general.\\
There are 3 singular orbit closures, indicated by red colour: 
\begin{itemize}
\item $Z=Z(s_1 s_3 s_2 , s_1 )$: We have $s_1 s_3 s_2 o_2 = (4,2,3,1) \in W$. Since $X((4,2,3,1))$ $\subseteq G/B$ is a singular Schubert variety (Remark \ref{rmrPattern}), $X_P (s_1 s_3 s_2) \subseteq G/P_2$ is singular, so $Z$ is singular by Proposition \ref{prpSmooth}. 
\item $Z=Z(s_2 s_1 s_3 s_2 , \text{id})$: One computes, that $t_k (s_2 s_1 s_3 s_2 , \text{id})$ equals $\phi^+ (C_2)$, so by Corollary \ref{crltangent_T_k} $T_p (Z) = T_p (G/C_2 ) = K^8$. So, $Z$ is singular since $\dim Z = l(s_2 s_1 s_3 s_2 ) + \dim Y^0 = 7$.
\item $Z=Z(s_1 s_3 s_2 ,\text{id})$: We compute $t_k (s_1 s_3 s_2 , \text{id}) = \phi^+ (C_k ) - \{(1,3),(2,4)\}$. So $T_k  (s_1 s_3 s_2 , \text{id})$ is $6$-dimensional. Set $v:= E_{1,1}- E_{3,3}$, $w:= E_{2,2} - E_{4,4}$. Then by Proposition \ref{prpcurvecontainment} $T_p (C(1,3))= K v$, $T_p (C(2,4))= Kw$. One computes now that $B_k (s_1 s_3 s_2 ,\text{id})$ as a vector space is generated by $T_k (s_1 s_3 s_2 , \text{id})$ and $K(v-w)$, and therefore is $7$-dimensional. So, $Z$ is singular since $\dim Z = 6$. (Using equations for the variety $Z$ one may compute that $B_k (\sigma , \alpha)=T_p (Z(\sigma , \alpha))$). 
\end{itemize}
The remaining $B$-orbit closures $Z(\sigma , \alpha)$ are all smooth and it holds that 
\begin{formula}\label{tangent}
$$\boxed{T_k (\sigma , \alpha ) = T_p (Z(\sigma , \alpha )).}$$
\end{formula}
The case of $Z=Z(s_1 s_3 s_2 , \text{id})$ shows that in general the formula \ref{tangent} is false. (Instead, in this example it holds that $B_k  (\sigma ,\alpha)=T_p (Z(\sigma ,\alpha))$).

\medskip

Nonetheless, in the next section we will see that results of L. Fresse \cite{FressePhd} imply that formula \ref{tangent} holds whenever $Z(\sigma , \alpha)$ is contained in $\mathfrak{n}$, where $\mathfrak{n} \subseteq\mathfrak{gl}_n$ is the affine space of upper-triangular matrices.

\end{exm}

\section{Comparison to results on components of Springer fibers}

\begin{rmr}\label{rmrSpringerResolution}
Denote by $\mathcal{N} \subseteq\mathfrak{gl}_n$ the nilpotent cone of $\mathfrak{gl}_n$, and by 
$$\tilde{\mathcal{N}} := \{ (x , V^{\bullet}) \in\mathcal{N}\times G/B \mid \forall i\geq 1:\quad  x(V^i ) \subseteq V^{i-1} \} \quad (V^0 := \{ 0 \}).$$
Then the map $\pi: \tilde{\mathcal{N}} \to \mathcal{N},\quad (x, V^{\bullet}) \mapsto x$ is a resolution of singularities called \textit{Springer resolution of} $\mathcal{N}$ \cite[Proposition 3.2.14, Definition 3.2.4 ]{Chriss-Ginzburg}. The fibre of $x_k$, $\mathcal{F}_k := \pi^{-1}(x_k )$, identifies with a $C_k$-stable subscheme of $G/B$. Its irreducible components are the $C_k$-orbit closures in $\mathcal{F}_k$. We will refer to them as \textit{Springer fiber components in the $2$-column case}, since the Young diagram associated with the nilpotent orbit $\mathbb{O}_k$ has two columns.
\end{rmr}

\begin{dfnrmr}\label{dfntypesing}
Let $(X,x)$ and $(Y,y)$ be pointed $K$-varieties. $(X,x)$ and $(Y,y)$ are said to be smoothly equivalent if there exists a pointed $K$-variety $(Z,z)$ and smooth morphisms $f: Z \to X$ and $g : Z \to Y$ such that $f(z)=x$ and $g(z)=y$. Using that smoothness of a morphism is stable under base change, one obtains that smooth equivalence is an equivalence relation on the set of pointed $K$-varieties. The equivalence classes are called types of singularities. In the sequel, we take advantage of \cite[Lemma 5.16]{Bongartz}: If an algebraic group $H$ acts on a variety $Z$, then $(Z,z)$ and $(Z ,h.z)$ are smoothly equivalent, for all $(h,z) \in H \times Z$, as $Z \to Z , z \mapsto h.z$ is an isomorphism. Furthermore, if $H \subseteq G$ is a subgroup, the map
$$\psi: H \backslash Z \to G \backslash (G \times_H Z) , \quad H.z \mapsto G.[1,z] = G \times_H H.z$$
is a bijection that preserves orbit closures. It preserves codimensions, because $\text{Stab}_H (z) = \text{Stab}_G ([1,z])$. Now, for an $H$-orbit closure $Z' \subseteq Z$, consider the morphisms
\begin{align*}
Z' \xleftarrow{p} G \times Z' \xrightarrow{q} G \times_H Z' , \quad p(g,z)=z , \quad q(g,z) = [g,z].
\end{align*}
As $G$ is smooth, $p$ is a $H$-equivariant smooth morphism. Now consider the cartesian square
$$
\begin{diagram}
\node{G \times Z'} \arrow{s,l}{q}\arrow{e,t}{} \node{G}\arrow{s,r}{\pi}\\
\node{G \times_H Z'} \arrow{e,t}{} \node{G/H}
\end{diagram}.
$$
Since $G=\text{GL}_n$ is a special group, $\pi$ is a Zariski locally trivial $H$-fibration and therefore a smooth morphism. Hence, $q$ is a $G$-equivariant smooth morphism. We conclude, that $\psi$ preserves the singularity types of orbit closures.
\end{dfnrmr}

\begin{lmm}\label{lmmcorrespondence}
The map 
$$B\backslash\mathbb{O}_k \to C_k \backslash G/B , \quad \mathcal{C}(\sigma , \alpha ) \mapsto (C_k \dot\alpha^{-1}\dot\sigma^{-1} B)/B$$
is a bijection that preserves codimensions, orbit closures and their types of singularities.
\end{lmm}

\begin{proof}
Let $G$ act on $G/C_k \times G/B$ diagonally. Consider the isomorphisms of $G$-varieties
$$\begin{diagram}
\node{G \times_{B} G/C_k} \arrow{e,t}{\sim} \node{G/B \times G/C_k} \arrow{e,t}{\sim} \node{G \times_{C_k} G/B}\\
\node{[g,h C_k ]}             \arrow{e,t}{\sim} \node{(g B , gh C_k)}   \arrow{e,t}{\sim}  \node{[gh , h^{-1} B]}  
\end{diagram}$$
Now, consider the two associated fibre bundles $\pi_1$ and $\pi_2$ given by
$$G/B \mathop{\longleftarrow}^{\pi_1} G \times_{B} G/C_k \simeq G/B \times G/C_k  \simeq G \times_{C_k} G/B \mathop{\longrightarrow}^{\pi_2} G/C_k .$$
By applying Remark \ref{dfntypesing} we obtain a bijection
$$B \backslash \mathbb{O}_k \to C_k \backslash G/ C_k , \quad B g C_k /C_k \mapsto C_k g^{-1} B / B ,$$
that preserves codimensions, orbit closures and their types of singularities. Since $Z_k \times W_k$ parametrizes $B \backslash \mathbb{O}_k$, the claim follows.
\end{proof}

\begin{dfn}\label{dfnC_kOrbitClosure} 
For a pair $(\sigma , \alpha) \in Z_k \times W_k$ we denote
$$S(\alpha^{-1}, \sigma^{-1}) := \overline{C_k \dot\alpha^{-1}\dot\sigma^{-1}B/B} \subseteq G/B .$$
By Lemma \ref{lmmcorrespondence} these are all $C_k$-orbit closures. Furthermore we denote $q = \text{id}B$.
\end{dfn}

\begin{lmm}\label{lmmTangentCorrespondence} Denote by $\mathfrak{b}$ (resp. by $\mathfrak{c}_k$) the Lie algebra of $B$ (resp. of $C_k$). Then, for all $(\sigma , \alpha) \in Z_k \times W_k$ there is an isomorphism of $K$-vector spaces
$$\mathfrak{g}/\mathfrak{b} \oplus T_p (Z(\sigma , \alpha)) \simeq \mathfrak{g}/ \mathfrak{c}_k \oplus T_q (S(\alpha^{-1},\sigma^{-1})).$$
\end{lmm}

\begin{proof} Let $(\sigma,\alpha)\in Z_k \times W_k$. The isomorphism of $G$-varieties $G \times_{C_k} G/B \to G \times_B G/C_k$ given in the proof of Lemma \ref{lmmcorrespondence} restricts to an isomorphism of $G$-varieties
$$G \times_{C_k} S({\alpha}^{-1} , {\sigma}^{-1}) \xrightarrow{\sim} G \times_B Z(\sigma,\alpha) , \quad [\text{id},q] \mapsto [\text{id}, p].$$
Because $B$ is a solvable group, $G \times_B Z(\sigma , \alpha) \to G/B$ is a Zariski locally trivial $Z(\sigma,\alpha)$-fibration. Since $G=\text{GL}_n$ is special, the fibre bundle $G \times_{C_k} S(\alpha^{-1},\sigma^{-1}) \to G/C_k$ is a Zariski locally trivial $S(\alpha^{-1},\sigma^{-1})$-fibration. Therefore, the claim follows. 
\end{proof}

\paragraph{Comparison with work of Perrin/Smirnov \cite{Smirnov-Perrin}.\\}

\noindent In \cite[Theorem 1.1, Corollary 1.4]{Smirnov-Perrin} it is shown that Springer fibre components in the two-column case have at most rational singularities. Furthermore, they remain normal in positive characteristic.\\
By using Lemma \ref{lmmcorrespondence}, Theorem \ref{thrRational} gives (for $\text{char}(K)=0$) an extension to the situation of an arbitrary $C_k$-orbit closure in $G/B$:

\begin{thr}\label{thrPerrinSmirnoff}
The $C_k$-orbit closures in $G/B$ have rational singularities.\\
Hence also the irreducible components of $\mathcal{F}_k$ have  rational singularities. 
\end{thr}

\begin{proof}The bijection from Lemma \ref{lmmcorrespondence} preserves types of singularities. Now, smooth equivalence preserves rationality of singularities (cf. \cite[Corollary 3.2]{SkowronskiZwara}).
\end{proof}

\paragraph{Comparison with work of Fresse \cite{FressePhd}, \cite{Fresse}.}

\begin{lmm}\label{lmmuppertriang}
$(C_k (\dot\sigma\dot\alpha)^{-1} B)/B \subseteq \mathcal{F}_k$ if and only if $(\dot\sigma\dot\alpha) . x_k$ is a strictly upper-triangular matrix.  
\end{lmm}

\begin{proof}
Recall that $(\dot\sigma\dot\alpha) . x_k = (\dot\sigma\dot\alpha) \cdot x_k \cdot(\dot\sigma\dot\alpha)^{-1} \in\mathbb{O}_k$. The claim now follows from the Definition of the Springer fiber $\mathcal{F}_k$.
\end{proof}

\begin{rmr}\label{rmrOrbital}
By Lemma \ref{lmmuppertriang} the correspondence from Lemma \ref{lmmcorrespondence} restricts to a bijection $$B\backslash(\mathbb{O}_k \cap\mathfrak{n}) \mathop{\longleftrightarrow}^{1:1} C_k \backslash \mathcal{F}_k ,$$
where $\mathfrak{n} \subseteq\mathfrak{gl}_n$ is the Lie algebra of strictly upper-triangular matrices (cf. \cite[Proposition 3.2]{FresseMelnikov13} where arbitrary nilpotent orbits of $\mathfrak{gl}_n$ are considered). The irreducible components of $\mathbb{O}_k \cap\mathfrak{n}$ are called \textit{orbital varieties of the nilpotent orbit} $\mathbb{O}_k$. They are among the $B$-orbit closures in $\mathbb{O}_k \cap\mathfrak{n}$. The irreducible components of $\mathcal{F}_k$(and therefore the orbital varieties of $\mathbb{O}_k$) are known to be equidimensional. (cf. \cite{Spaltenstein} where the case of an arbitrary nilpotent orbit is considered). They are indexed by numberings $T$ of the Young diagram of the nilpotent orbit $Y$ with $\{1,\ldots ,n\}$ such that numbers in each row increase from left to right and in each column from the top to the bottom. One calls these numberings \textit{standard Young tableaux of shape} $Y$ and denotes the corresponding Springer fiber component by $
 K^T$. In our case, the Young diagram of $\mathbb{O}_k$ has two columns of length $n-k$ and $k$ (the \textit{$2$-column case}) and the components of $\mathcal{F}_k$ are of dimension 
$$\frac{k (k-1) + (n-k)(n-k-1)}{2}.$$
\end{rmr}
L. Fresse established in \cite{FressePhd} a singularity criterion for Springer fiber components in the $2$-column-case. Recently, in \cite{Fresse} he computed the tangent spaces of $C_k$-orbit closures $S(\alpha^{-1},\sigma^{-1}) \subseteq \mathcal{F}_k$ at arbitrary points. This will allow us to show that formula \ref{tangent} holds in the case where a $B$-orbit closure is contained in $\mathfrak{n}$. In contrast to our normal form given by $Z_k \times W_k$, L. Fresse uses A. Melnikov's link patterns ( cf. \cite{Melnikov} or Remark \ref{rmrolps}). So, there is something left to translate for us. 
\begin{dfn}\label{dfnInvolutionsTauFlags}
Define $S_n^2 (k) \subset S_n$ as the set of involutions fixing exactly $n-2k$ elements of $\{1,\ldots ,n\}$, and $(Z_k \times W_k )^{\text{up}}$ as the subset of $Z_k \times W_k$ consisting of those $(\sigma,\alpha)$ such that $\dot \sigma \dot\alpha . x_k$ is upper-triangular. We then have a bijection
$$(Z_k \times W_k )^{\text{up}} \to S_n^2 (k), \quad (\sigma,\alpha)\mapsto \tau_{\sigma\alpha}:=\prod_{i=1}^k (\sigma \alpha (i) \quad \sigma (n-k+i) ).$$
Let $\tau \in S_n^2 (k)$. A complete flag $((f_1 , \ldots , f_s ))_{s = 1, \ldots ,n}$ is said to be a $\tau$-flag if
$$x_k (f_i ) = 
\begin{cases} 
f_{\tau(i)}, & \mbox{if } \tau(i)< i ,\\
0, & \mbox{else.}
\end{cases}$$
We denote the set of $\tau$-flags by $\mathcal{Z}_{\tau}$. Its closure in $G/B$ is denoted by $\overline{\mathcal{Z}_{\tau}}$.
\end{dfn}

\begin{rmr}\label{rmrC_kFlagsOrbits}(cf. \cite[Remark 2]{Fresse})
Let $(\sigma , \alpha) \in (Z_k \times W_k )^{\text{up}}$. It holds that
$$\mathcal{Z}_{\tau_{\sigma\alpha}} = C_k \dot\alpha^{-1}\dot\sigma^{-1} B/B .$$
Consequently, we have $\overline{\mathcal{Z}_{\tau_{\sigma\alpha}}} = S(\alpha^{-1} ,\sigma^{-1}) \subseteq G/B$.
\end{rmr}
Rephrasing the definition of \textit{adjacent link patterns} for the special case of a link pattern adjacenct to the minimal link pattern, we obtain 
\begin{dfn}\label{dfnAdjacency}(cf. \cite[Definition 3]{Fresse}) Let $(\sigma , \alpha) , (\sigma' , \alpha') \in Z_k \times W_k$.
\begin{itemize}
\item $\phi_{\mathfrak{n}}^+ (C_k ) := \phi^+ (C_k ) - \{(i,j)\in\{1, \ldots ,k\} \times \{n-k+1,\ldots ,n\}\mid i+n-k \geq j\}$.
\item $\tau_{\sigma' \alpha'} \in S_n^2 (k)$ is said to be adjacent to $\tau_{\text{id}}$ if $\sigma' \alpha' W(C_k ) = \epsilon W(C_k)$, for some 
$\epsilon \in  \phi_{\mathfrak{n}}^+ (C_k )$.
\item $s(\sigma ,\alpha ):= \{\epsilon W(C_k) \mid \epsilon W(C_k ) \prec \sigma\alpha W(C_k ) \text{ and } \tau_{\epsilon} \text{ is adjacent to } \tau_{\text{id}} \}.$
\end{itemize}
\end{dfn}
Now, we can see that another version of (a part of) \cite[Theorem 1]{Fresse} is to say that formula \ref{tangent} holds if a $B$-orbit closure $Z(\sigma,\alpha)$ is contained in $\mathfrak{n}$:  
\begin{crl}\label{thrUppercaseTangent}
Let $(\sigma , \alpha) \in Z_k \times W_k$ such that $Z(\sigma , \alpha) \subseteq \mathfrak{n}$. Then 
\begin{enumerate}
\item $T_k  (\sigma , \alpha) = T_p (Z(\sigma , \alpha))$.
\item $Z(\sigma , \alpha)$ is regular iff $\mid t_k (\sigma , \alpha )\mid = l(\sigma)+l(\alpha)$. 
\end{enumerate}
\end{crl}

\begin{proof}
Statement \textit{2.} follows from statement \textit{1.}, since 
$$\dim T_k (\sigma , \alpha ) = \dim Y^0 + \mid t_k (\sigma , \alpha ) \mid ,\quad \dim Z(\sigma , \alpha) = \dim Y^0 + l(\sigma) + l(\alpha).$$ 
Now, in order to proof Statement \textit{1.} note that $Z(\sigma , \alpha ) \subseteq \mathfrak{n}$ implies $t_k (\sigma , \alpha) \subseteq \phi_{\mathfrak{n}}^+ (C_k )$. (cf. Proposition \ref{prpcurvecontainment}, part \textit{2.}). We have thus 
$$s(\sigma , \alpha) = t_k (\sigma , \alpha).$$
\cite[Theorem 1, (a)]{Fresse} states that
$$\dim T_q (S(\alpha^{-1}, \sigma^{-1})) = \dim (C_k / B \cap C_k ) + \mid s(\sigma , \alpha) \mid .$$
Using $Y^0 = B/B \cap C_k$ and Lemma \ref{lmmTangentCorrespondence} we obtain
$$\dim T_p (Z(\sigma , \alpha)) = \dim Y^0 + \mid s(\sigma , \alpha) \mid = \dim Y^0 + \mid t_k (\sigma , \alpha) \mid = \dim T_k (\sigma , \alpha),$$
showing that the claim is true.   
\end{proof}

\begin{rmr}\label{rmrRowStandard}
A numbering of a Young diagram is said to be \textit{row-standard} if the numbers increase in each row from left to right. Due to Lemma \ref{lmmuppertriang} and Remark \ref{rmrolps} we can deduce from Theorem \ref{throrbits} that $B$-orbit closures in $\mathbb{O}_k \cap n$ are indexed by row-standard tableaux
$$T_\tau =\tiny{\ytableausetup{mathmode , boxsize=4.5em}
\begin{ytableau}
\tau_1 & \tau_{n-k+1} \\
 \vdots & \vdots \\
 \tau_{k} & \tau_{n} \\
 \vdots  \\
 \tau_{n-k} 
\end{ytableau}}\quad ,$$
where $\tau = \sigma \alpha$, for a pair $(\sigma , \alpha) \in Z_k \times W_k$: $T_\tau$ is row-standard iff $\dot\tau . x_k \in\mathfrak{n}$.\\
By the description of $Z_k$ given in Remark \ref{rmrPattern} we then obtain that $B$-orbit closures in $\mathbb{O}_k \cap\mathfrak{n}$ are parametrized by those row-standard tableaux $T_\tau$ where the numbers in the second column increase from the top to the bottom, and in addition $\tau_{k+1} < \ldots < \tau_{n-k}$ if $2k < n$.\\
A standard tableau $T$ corresponds to an irreducible component $K^T$ of $\mathcal{F}_k$. In \cite[Discussion after Lemma 2.2]{FressePhd} it is shown how to obtain from a standard tableau $T$ a row-standard tableau $T_{\sigma\alpha}$ by renumbering the labels in the first column of $T$, such that the $C_k$-orbit closure $ S(\alpha^{-1},\sigma^{-1})$ equals the irreducible component $K^T$. We will demonstrate this and Corollary \ref{thrUppercaseTangent} in the next example. 
\end{rmr}

\begin{exm}\label{exm3}
Let $(n,k)=(6,2)$. We consider a standard tableau $T$ with corresponding singular Springer fiber component $K^T$. The discussion before \cite[Lemma 2.2]{FressePhd} shows how to obtain a row-standard tableau $T_{\sigma\alpha}$, for $(\sigma ,\alpha)\in Z_2 \times W_2$ such that $Z(\sigma , \alpha)$ equals the orbital variety of $\mathbb{O}_k$ corresponding to $K^T$, under the bijection from Lemma \ref{lmmcorrespondence}.
$$T =\tiny{\ytableausetup{mathmode , boxsize=2em}
\begin{ytableau}
1&3\\
2&5\\
4\\
6
\end{ytableau} \quad\rightsquigarrow\quad  
T_{\sigma\alpha} =\ytableausetup{mathmode , boxsize=2em}
\begin{ytableau}
2&3\\
4&5\\
1\\
6
\end{ytableau} \quad\rightsquigarrow\quad \sigma = (2,4,1,6,3,5), \quad\alpha = \mathrm{id}}.$$
We compute $\sigma = s_1 s_3 s_2 s_5 s_4$ and  $W(C_k )= <s_1 s_5 , s_3 >$. Denote by $\equiv$ congruence modulo $W(C_k )$. In the next tabular we check whether a transposition from $\phi^+ (C_k )$ belongs to $t_k (\sigma , \alpha)$ or not.

\bigskip

\begin{tabular}{|c|c|c|}
\hline
$(12)= s_1 \prec \sigma$ & $(13) \equiv s_1 s_2 s_5 \prec \sigma$ & $(14) \equiv s_3 s_1 s_2 s_5 \prec \sigma$ \\
\hline
$(15)\notin \phi_{\mathfrak{n}}^+ (C_k )$ & $(16) \not\equiv \tau\prec\sigma$ & $(23) = s_2 \prec\sigma$ \\
\hline
$(24) \equiv s_3 s_2 \prec \sigma$ & $(25) \notin \phi_{\mathfrak{n}}^+ (C_k )$ & $(26) \notin \phi_{\mathfrak{n}}^+ (C_k )$\\
\hline
$(35) \equiv s_3 s_4 \prec \sigma$ & $(36) \equiv s_1 s_3 s_5 s_4 \prec \sigma$ & $(45) = s_4 \prec \sigma$ \\
\hline
$(46) \equiv s_1 s_5 s_4 \prec \sigma$ & &\\
\hline
\end{tabular} 

\bigskip

So $t_k (\sigma , \text{id})$ has $9 > l(\sigma) = 5$ elements wherefore $Z(\sigma , \text{id})$ is singular.  This was the first example of a singular Springer fiber component in the two-column case.\cite{Vargas}. For $(n,k)=(6,2)$ it is the only singular component of $\mathcal{F}_k$ \cite{FressePhd}.    
\end{exm}

\section{Concluding Questions}

Let $Z=Z(\sigma , \alpha)$ and $\mathfrak{Z} = \mathfrak{Z}(\sigma , \alpha)$.

\begin{itemize}
\item Computing $t_k (\sigma , \alpha)$ is a bad job. Is there a \textit{pattern avoidance} singularity criterion for $Z \subseteq \mathfrak{n}$? Could it be possibly related to Fresse's and Melnikov's singularity criterion for orbital varieties of $\mathbb{O}_k$ given in terms of \textit{minimal arcs} of link patterns \cite[Theorem 5.2]{FresseMelnikov13}?
\item \cite[Theorem 1]{Fresse} actually describes all tangent spaces of $S(\alpha^{-1}, \sigma^{-1})$, making it possible to determine the singular locus of $S(\alpha^{-1},\sigma^{-1})$. Can Corollary \ref{thrUppercaseTangent} be extended to tangent spaces at arbitrary points by considering translates of $T_k$-lines $C(\epsilon)$, where $\epsilon \in \phi^+ (C_k )$ ?
\item In Example \ref{exm2} we have seen, that the tangent spaces of $T_k$-lines contained in $Z$ do not span $T_p (Z)$, if $Z \not\subseteq \mathfrak{n}$. Do they at least generate $T_p (Z)$ as a $(B \cap C_k )$-module? 
\item The Bruhat order on $B \backslash \mathfrak{Z}$ has been determined by Boos and Reineke \cite{Boos-Reineke}. Theorem \ref{thrBottSamelsonSpringer} shows how to resolve the singularities of $\mathfrak{Z}$.  What are reduced equations for $\mathfrak{Z}$? When is $\mathfrak{Z}$ singular? When is $\mathfrak{Z}$ normal?
\item What can be said about the geometry of $B$-orbit closures in spherical nilpotent orbits of Lie algebras which are different from $\mathfrak{gl}_n$? Does one get more examples of spherical homogeneous spaces \textit{of minimal rank} in the sense of N. Ressayre \cite{Ressayre09}? 
\end{itemize}

\end{document}